\RequirePackage{ifpdf}
\ifpdf 
\documentclass[pdftex]{sigma}
\else
\documentclass{sigma}
\fi

\usepackage{amsmath}
\usepackage{amsfonts}
\usepackage{amssymb}
\usepackage[mathscr]{eucal} 
\usepackage{graphicx}
\usepackage{hyperref}
\usepackage[new]{old-arrows}
\usepackage{enumitem}
\usepackage{comment}
\usepackage{todonotes}
\usepackage[all]{xy}
\usepackage[pagewise]{lineno}

\usepackage{xparse} 

\makeatletter
\newcommand{\Xbar}{{\mathchoice
     {\smash@bar\textfont\displaystyle{0.55}{2.5}\mathscr{X}}
     {\smash@bar\textfont\textstyle{0.55}{2.5}\mathscr{X}}
     {\smash@bar\scriptfont\scriptstyle{0.55}{2.5}\mathscr{X}}
     {\smash@bar\scriptscriptfont\scriptscriptstyle{0.55}{2.5}\mathscr{X}}
          }}
\newcommand{\smash@bar}[4]{
     \smash{\rlap{\raisebox{-#3\fontdimen5#10}{$\m@th#2\mkern#4mu\mathchar'26$}}}          }
\makeatother
\newcommand{\X}[1]{\mathscr{\Xbar}_{#1}}

\newcommand{\D}{\mathscr{D}}
\newcommand{\K}{\mathscr{K}}
\newcommand{\LL}{\mathscr{L}}
\NewDocumentCommand{\R}{G{}}{\IfNoValueTF{#1}{\mathbb{R}}{\mathbb{R}^{#1}}}
\newcommand{\T}{\mathsf{T}}

\newcommand{\OO}{\mathscr{O}}
\newcommand{\PP}{\mathscr{P}}

\newcommand{\Cinf}[1]{\mathbf{\mathit{C}}^{\infty}_{#1}}
\newcommand{\cSch}[1]{[\hspace{-0.065cm}[ #1 ]\hspace{-0.065cm}]}
\newcommand{\dd}{\mathrm{d}}
\newcommand{\dlie}[1]{\mathrm{L}_{#1}}
\newcommand{\ec}[1]{\mbox{$#1$}}
\newcommand{\ii}{\mathbf{i}}

\newcommand{\longtwoheadrightarrow}{\relbar\joinrel\twoheadrightarrow}

\newcommand{\orcid}[1]{\href{https://orcid.org/#1}{\includegraphics[width=0.02\textwidth]{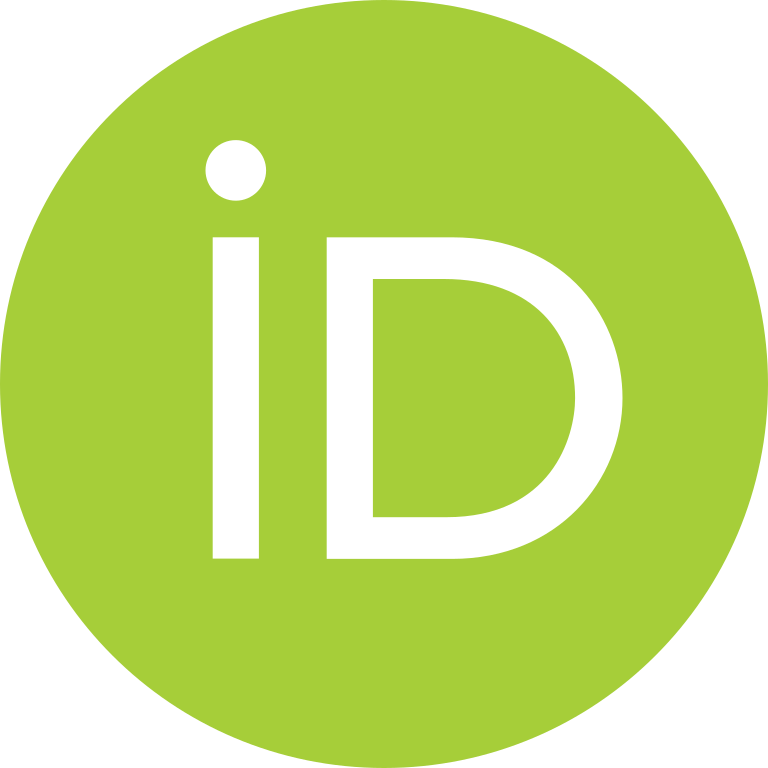}}}

\usepackage{amsthm}
\usepackage{thmtools}

\numberwithin{equation}{section}

\begin{document}

\renewcommand{\PaperNumber}{***}

\thispagestyle{empty}


\ArticleName{On the First Cohomology of Infinitesimal Poisson Algebras} 
\ShortArticleName{Vanishing of the Infinitesimal First Cohomology of Poisson Submanifolds}

\Author{D. Garc\'ia-Beltr\'an~$^{\dag}$
J. C. Ru\'iz-Pantale\'on~$^{{\includegraphics[width=0.015\textwidth]{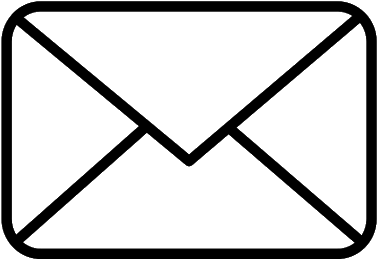}}\,\dag}$ 
and E. Velasco-Barreras~$^{\dag}$
}
\AuthorNameForHeading{Garc\'ia-Beltr\'an, Ru\'iz-Pantale\'on, Velasco-Barreras}

\Address{$^{\dag}$~Departamento de Matem\'aticas, Universidad de Sonora}
\Address{\phantom{$^{\dag}$~}Blvd. Luis Encinas y Rosales, Col. Centro, Hermosillo, Sonora, México, 83000}
\EmailD{$^{1}$\href{mailto:email@address}{dennise.garcia\,@unison.mx},
        $^{{\includegraphics[width=0.015\textwidth]{sobre.png}}}$\href{mailto:email@address}{jose.ruiz\,@unison.mx},
        $^{2}$\href{mailto:email@address}{eduardo.velasco\,@unison.mx}
        }


\Abstract{For the so-called infinitesimal Poisson algebras encoding first-order jets of Poisson submanifolds, we provide a description of their first cohomology in terms of intrinsic cohomologies of the underlying Poisson submanifold. We establish a natural mapping from their first cohomology to the first Poisson cohomology of the corresponding Poisson submanifold. Moreover, we formulate necessary and sufficient conditions for the vanishing of the first cohomology of infinitesimal Poisson algebras. Finally, we consider the special cases of symplectic leaves and, more generally, Poisson submanifolds with partially split first-order jets. In particular, we derive cohomological necessary conditions for the partially split property.}



\Classification{53D17, 17B63, 17B56, 53C05, 17B60, 16W25}

    \section{Introduction}

Understanding the cohomology of Poisson algebras is a fundamental problem in Poisson Geometry as it provides information about derivations, deformations, and certain invariants of Poisson brackets (see, for example, \cite{Duf05,Laur2013} and reference therein). From a geometric perspective, on a Poisson manifold, these cohomologies often encode obstructions to various local or semi-local properties, such as linearizability and rigidity around Poisson submanifolds, as well as global properties such as unimodularity \cite{We97} and strong properness (see \cite[Proposition 7.3]{pmct1}). They can also be used to distinguish inequivalent Poisson manifolds under gauge or Morita equivalence \cite{GinzMorita}. However, computing these cohomologies is typically a very difficult problem (see, for example, \cite{Va94,Duf05,Laur2013,CraFeMa21} and reference therein).

On the other hand, the linearization problem for Poisson manifolds naturally leads to questions regarding the existence of linear models for Poisson structures in neighborhoods of distinguished submanifolds \cite{Wei83, Conn85, Duf05, CraFeMa21}. Linear models for Poisson bivector fields are known to exist around symplectic leaves \cite{Vo01}, Poisson transversals \cite{FrMNormal}, and Poisson submanifolds with partially split first jet \cite{FerMar22}. However, such models do not generally exist for arbitrary Poisson submanifolds (see \cite[Section 9]{FerMar22}).

Although Poisson manifolds may not admit linear models around arbitrary Poisson submanifolds, there always exists a first-order jet. In the literature, the following equivalent approaches to this notion are considered (see, \cite[Proposition 5.1.8]{Mar13}):
    \begin{enumerate}[label=(\arabic*)]
        \item A Poisson structure on the algebra of fiberwise affine functions on the normal bundle, compatible with the Lie algebra of co-normal sections and the Poisson structure of the submanifold \cite{Mar12, RuGaVo20} (see also \cite{GaRuVo23}). \label{item:Approach1}
        \item A Lie algebroid extension of the cotangent Lie algebroid of the submanifold whose kernel is the co-normal bundle \cite{Mar12}. This can be also formulated in terms of infinitesimally multiplicative closed 2-forms \cite{FerMar22}.
        \item An equivalence class of bivector fields satisfying the Jacobi identity at first-order \cite{FerMar22}.
    \end{enumerate}

Given a Poisson submanifold, approach \ref{item:Approach1} leads to a Poisson algebra that, via an exponential map, coincides with the bracket of the ambient Poisson manifold at first-order. This Poisson algebra is what we call the \emph{infinitesimal Poisson algebra} of the Poisson submanifold \cite{RuGaVo20}.

Motivated by this observation, the purpose of this paper is to study certain aspects of the first cohomology \ec{\mathrm{H}^{1}(\PP)} of the infinitesimal Poisson algebra $\PP$ of a Poisson submanifold \ec{(S,\psi)}. Specifically:
    \begin{itemize}
        \item Express \ec{\mathrm{H}^{1}(\PP)} in terms of the cohomology of \ec{(S,\psi)} and the Lie algebra cohomology of the co-normal bundle $E^{\ast}$ of $S$.
        \item Determine whether there exists a relation between \ec{\mathrm{H}^{1}(\PP)} and the first Poisson cohomology of \ec{(S,\psi)}.
        \item Provide necessary and sufficient conditions for the vanishing of \ec{\mathrm{H}^{1}(\PP)}.
    \end{itemize}

Our approach is based on the fact that the Poisson bracket of $\PP$ can be written in terms of the data \ec{([\cdot,\cdot]_1,\D,\K)} consisting of the Lie bracket of sections \ec{[\cdot,\cdot]_1} on \ec{\Gamma E^\ast}, a contravariant derivative $\D$ on $E^{\ast}$ and bivector field $\K$ on $S$ with values on $E^{\ast}$. These data satisfy compatibility conditions that encode a factorization of the Jacobi identity.

Building upon the theoretical framework of \cite{RuGaVo20,GaRuVo23}, we derive the following results: first, Theorem \ref{teo:ShortExactSeq} presents a splitting of the infinitesimal Poisson algebra of \ec{(S,\psi)},
    \begin{equation}\label{eq:IntroSplit}
        \mathrm{H}^{1}(\PP) \simeq \frac{\operatorname{Im}{\sigma|_{\mathfrak{M}(\PP)}}}{\mathrm{Ham}(S,\psi)} \oplus \frac{\mathrm{H}_{\partial_{\D}}^{1}}{\ker\mathrm{J}} \oplus \frac{\mathfrak{M}_{0}(\PP)}{\mathscr{C}_{0}(\PP) + \operatorname{Inn}{\mathscr{G}}},
    \end{equation}
where each component corresponds to intrinsic cohomologies associated with the Poisson submanifold:
\begin{itemize}
    \item The first factor in the splitting in \eqref{eq:IntroSplit} is a Lie subalgebra of the first Poisson cohomology of \ec{(S,\psi)} that consists of the cohomology clases of Poisson vector fields of $\psi$ which, as derivations of the algebra $\Cinf{S}$, can be extended to a Poisson derivation of $\PP$, 
        \begin{equation*}
            \frac{\operatorname{Im}{\sigma|_{\mathfrak{M}(\PP)}}}{\mathrm{Ham}(S,\psi)} \subseteq \mathrm{H}^1(S,\psi).
        \end{equation*}
    \item \ec{\mathrm{H}_{\partial_{\D}}^{1}} is the first cohomology of the cochain complex consisting of the module of multivector fields on $S$ with values in the center of the Lie algebra \ec{(\Gamma{E^{\ast}},[\cdot,\cdot]_{1})}, and the coboundary operator \ec{\partial_{\D}} given by the restriction of the contravariant differential induced by $\D$ to this module.
    \item \ec{\mathfrak{M}_{0}(\PP) / (\mathscr{C}_{0}(\PP) + \operatorname{Inn}{\mathscr{G}})} is a Lie subalgebra of the first Lie algebra cohomology of  \ec{(\Gamma{E^{\ast}},[\cdot,\cdot]_{1})} with coefficients in the adjoint representation.
\end{itemize}
Moreover, Theorem \ref{teo:ShortExactSeq} generalizes the results of \cite[Section 10]{GaRuVo23} by providing a description of \ec{\mathrm{H}^{1}(\PP)} for arbitrary Poisson submanifolds, which extends the particular cases considered in \cite{GaRuVo23}.


By Theorem \ref{teo:ShortExactSeq}, there exists an $\R{}$-linear mapping from the first cohomology of the infinitesimal Poisson algebra of $S$ to its first Poisson cohomology,
    \begin{equation*}
        \zeta: \mathrm{H}^{1}(\PP) \to \mathrm{H}^1(S,\psi).
    \end{equation*}
A natural question is whether this mapping is surjective. This nontrivial problem has been studied in an algebraic framework for the case of Poisson algebras induced by Poisson modules \cite{Zhu2020}. Below, we provide positive answers to this question for two cases: \emph{symplectic leaves} and \emph{Poisson submanifolds with partially split firstorder jet}. To address the question, we derive some properties of the image of $\zeta$. 

Theorem \ref{teo:H1primprim} shows that the image of $\zeta$ contains the following Lie subalgebras of \ec{\mathrm{H}^1(S,\psi)}:
    \begin{enumerate}[label=(\alph*)]
        \item The image of the first de Rham cohomology of $S$ under the Lichnerowicz homomorphism $\mathrm{H}^{1}_{\mathrm{dR}}(S) \to \mathrm{H}^{1}(S,\psi)$ induced by $\psi$. \label{item:IntroTeo32A}
        \item A Lie subalgebra of the so-called \emph{cotangential} Poisson cohomology of \ec{(S,\psi)}, consisting of cohomology classes in \ec{\mathrm{H}^{1}(S,\psi)} represented by Poisson vector fields of the form \ec{\psi(\theta, \cdot)}, with \ec{\theta \in \Gamma\T^{\ast}S}.
        \item The cohomology classes in \ec{\mathrm{H}^{1}(S,\psi)} represented by the restrictions of Poisson vector fields on the ambient Poisson manifold to $S$. \label{item:IntroTeo32C}
    \end{enumerate}
Furthermore, the cotangential Poisson cohomology of \ec{(S,\psi)} lies in the image of $\zeta$ if the following condition holds: the natural inclusion \ec{\Cinf{S} \hookrightarrow \Cinf{\mathrm{aff}}(E) \simeq \Cinf{S} \oplus \Gamma{E^{\ast}}}, \ec{k \hookrightarrow  k \oplus 0}, induces an inclusion
    \begin{equation}\label{eq:IntroCasim}
        \mathrm{Casim}(S,\psi) \longhookrightarrow \mathrm{Casim}(\PP)
    \end{equation}
from the Lie algebra of Casimir functions of $\psi$ to the Lie algebra of Casimir elements of $\PP$. Here, \ec{\Cinf{\mathrm{aff}}(E)} is the commutative algebra of fiberwise affine functions on $E$. We note that this condition does not always holds. Lemma \ref{lema:CasimP} provides some equivalences.





Results above allows us to formulate necessary and sufficient conditions for the vanishing of the first cohomology of the infinitesimal Poisson algebra of $S$. Taking into account splitting \eqref{eq:IntroSplit}, Theorem \ref{teo:gen_main} states that it is sufficient for \ec{\mathrm{H}^{1}(\PP)=\{0\}} that the following three cohomologies intrinsic to the Poisson submanifold \ec{(S,\psi)} vanish:
    \begin{equation*}
        \mathrm{H}^{1}(S,\psi) = \{0\}, \quad \mathrm{H}_{\partial_{\D}}^{1} = \{0\} \quad \text{and} \quad \mathrm{H}^{1}(\mathscr{G},\mathscr{G}) = \{0\},
    \end{equation*}
Here, \ec{\mathrm{H}^{1}(\mathscr{G},\mathscr{G})} is the first cohomology of the Lie algebra \ec{\mathscr{G} = (\Gamma{E^{\ast}},[\cdot,\cdot]_{1})} with coefficients in the adjoint representation. In particular, as shown in Proposition \ref{prop:loctriv}, if the co-normal bundle of $(S,\psi)$ is a trivial Lie bundle with semisimple typical fiber and \ec{\mathrm{H}^{1}(S,\psi) = \{0\}}, then \ec{\mathrm{H}^{1}(\PP) = \{0\}}.



Taking into account items \ref{item:IntroTeo32A}-\ref{item:IntroTeo32C}, Theorem \ref{teo:VanishLich} states that for \ec{\mathrm{H}^{1}(\PP)=\{0\}}, it is necessary that the Lichnerowicz homomorphism induced by $\psi$ be the zero mapping, and that the restriction to \ec{(S,\psi)} of every Poisson vector field the ambient Poisson manifold that is tangent to $S$ be a Hamiltonian vector field of $\psi$. In particular, for regular Poisson submanifolds satisfying \eqref{eq:IntroCasim}, the vanishing of both the foliated de Rham cohomology of the symplectic foliation of \ec{(S,\psi)} and of the first cohomology \ec{\mathrm{H}_{\partial_{\D}}^{1}} are obstructions to the triviality of \ec{\mathrm{H}^{1}(\PP)}; see Corollary \ref{teo:VanishCasimInReg}.

\paragraph{The Symplectic Leaf Case.}
When the Poisson submanifold \ec{(S,\psi)} is symplectic, by Theorem \ref{teo:SurjSym}, the mapping $\zeta$ is surjective and
    \begin{equation*}
        \mathrm{H}^{1}(\PP) \simeq \mathrm{H}^{1}_{\mathrm{dR}}(S) \oplus \mathrm{H}_{\partial_{\D}}^{1}\oplus\frac{\mathfrak{M}_{0}(\PP)}{\operatorname{Inn}{\mathscr{G}}},
    \end{equation*}
where $\mathrm{H}^{1}_{\mathrm{dR}}(S)$ is the first de Rham cohomology of $S$. This result provide an infinitesimal version of those presented in \cite{VeVo18} for the description and vanishing of the germ of the first Poisson cohomology of a symplectic leaf.

In particular, for simply connected symplectic leaves with semisimple isotropy Lie algebras \cite{CraFeMa21}, the first cohomology of the infinitesimal Poisson algebra is trivial; see Proposition \ref{prop:SymplecticVanish}.

\paragraph{The Partially Split Case.}
If the Poisson submanifold \ec{(S,\psi)} admits a partially split first-order jet \cite{FerMar22}, then Theorem \ref{teo:PartiallySplit} establishes that the mapping $\zeta$ is surjective and
    \begin{equation*}
        \mathrm{H}^{1}(\PP) \simeq \mathrm{H}^1(S,\psi) \oplus \mathrm{H}_{\partial_{\D}}^{1} \oplus \frac{\mathfrak{M}_{0}(\PP)}{\operatorname{Inn}{\mathscr{G}}}.
    \end{equation*}
Moreover, condition \eqref{eq:IntroCasim} holds. 

Consequently, if the mapping $\zeta$ is not surjective for a Poisson submanifold \ec{(S,\psi)}, then its first-order jet is not partially split.

Finally, we outline the structure of the paper. In Section \ref{sec:BasicNotions}, we review basic notions and preliminary results, including the definitions of tangential and cotangential cohomologies of Poisson manifolds. Section \ref{sec:MainResults} presents the main results of this work. Since several proofs involve technical details, Section \ref{sec:proof} provides the corresponding complete proofs and explicit constructions.

    \section{Basic Notions}\label{sec:BasicNotions}

Here we recall some definitions, notions and facts that are used throughout this work.

\paragraph{Poisson Algebra Cohomology.}
An $R$-Poisson algebra \ec{\PP = (P,\{\cdot,\cdot\})} consists of an associative and commutative unital algebra $P$, over a commutative ring $R$ with unit, endowed with an $R$-Lie bracket \ec{\{\cdot,\cdot\}} compatible with the algebra product on $P$ by the Leibniz rule.

The \emph{cohomology} of $\PP$ is the cohomology of the cochain complex \ec{(\X{P}^{\bullet}:=\oplus_{k \in \mathbb{Z}}\,\X{P}^{k},\dd_{\PP})} \cite{Laur2013}, where \ec{\mathfrak{X}_{P}^{k}} is the $P$-module of $R$-multilinear skew-symmetric $k$-derivations \ec{X:P \times \cdots \times P \to P} and the coboundary operator is defined by
    \begin{multline*}
        \big( \dd_{\PP}X \big)(p_0,\dots,p_{k}) := \sum_{i=0}^{k}(-1)^{i} \big\{ p_{i},X(p_{0},\dots,\widehat{p}_{i},\dots,p_{k}) \big\} \\ 
        + \sum_{0\leq i < j\leq k}(-1)^{i+j} X\big( \{p_{i},p_{j}\},p_{0},\dots,\widehat{p}_{i},\dots,\widehat{p}_{j},\dots,p_{k} \big),
    \end{multline*}
for all \ec{X \in \X{P}^{k}} and \ec{p_0,\dots,p_{k} \in P}. Here, the symbol $\ \widehat{}\ $ denotes omission. 

    \newpage

In particular, the \emph{zeroth cohomology} is the center of the Poisson algebra, whose elements are called Casimir elements of $\PP$. The \emph{first cohomology} of $\PP$ is the quotient of the Lie algebra of Poisson derivations by the Lie ideal of Hamiltonian derivations,
    \begin{equation*}
        \mathrm{H}^{1}(\PP) = \frac{\mathrm{Poiss}(\PP)}{\mathrm{Ham}(\PP)}.
    \end{equation*}
A Poisson derivation of $\PP$ is a derivation $X$ of  the algebra $P$ satisfying \ec{X\{p_{1},p_{2}\} = \{X p_{1},p_{2}\} + \{p_{1},X p_{2}\}}, for \ec{p_{1},p_{2} \in P}. Given \ec{h \in P}, the derivation \ec{\{h, \cdot\}} is called Hamiltonian derivation of $\PP$, associated with the Hamiltonian element $h$.

The standard example of a Poisson algebra is given by the algebra of a Poisson manifold \ec{(S,\psi)}. In this case, \ec{P = \Cinf{S}} is the commutative algebra of smooth functions on $S$ with pointwise multiplication and \ec{\{f,g\}_{\psi} := \psi(\dd f, \dd g)} is the corresponding Poisson bracket on $S$ \cite{Duf05,Laur2013,CraFeMa21}.

\paragraph{Poisson Manifold Cohomology.}
Let \ec{(S,\psi)} be a smooth Poisson manifold. The \emph{Poisson cohomology} of \ec{(S,\psi)} is the cohomology of the cochain complex \ec{(\Gamma\wedge^{\bullet}\T{S},\dd_{\psi})} \cite{Li77}, induced by the coboundary operator
    \begin{equation*}
        \dd_{\psi} := \cSch{\psi,\cdot}.
    \end{equation*}
Here, \ec{\cSch{\cdot,\cdot}} is the Schouten-Nijenhuis bracket for multivector fields \cite{Sch40,Sch53,Nij55}. We remark that this definition is equivalent to the cohomology of the corresponding Poisson algebra \ec{(\Cinf{S},\{\,,\,\}_{\psi})}. In particular, the \emph{first Poisson cohomology} is the quotient of the Lie algebra of Poisson vector fields by the Lie ideal of Hamiltonian vector fields of $\psi$,
    \begin{equation*}
        \mathrm{H}^{1}(S,\psi) = \frac{\mathrm{Poiss}(S,\psi)}{\mathrm{Ham}(S,\psi)}.
    \end{equation*}
So, the Lie bracket of vector fields inherits a Lie algebra structure on \ec{\mathrm{H}^{1}(S,\psi)}.

\paragraph{Lie Subalgebras of \ec{\mathrm{\mathbf{H}}^{1}(S,\psi)}.}
A vector field $v$ on $S$ is said to be \emph{tangent} to the symplectic foliation $\mathcal{F}$ of $\psi$ if it is tangent to every symplectic leaf. If $\mathcal{F}$ is regular, then $v$ is in the image of the skew-symmetric mapping \ec{\psi^{\sharp}:\Gamma\T^{\ast}S \to \Gamma \T{S}}, \ec{\theta\mapsto\psi(\theta,\cdot)}. However, this may not occur in the general case. So, we introduce the following notion: a vector field $v$ on \ec{(S,\psi)} is said to be \emph{cotangent} if \ec{v=\psi^{\sharp}\theta}, for some \ec{\theta \in \Gamma\T^{\ast}S}. Moreover, we define the Lie algebra of \emph{cotangent Poisson vector fields} of \ec{(S,\psi)} by
    \begin{equation*}
        \mathrm{Poiss}_{\mathrm{cot}}(S,\psi) := \operatorname{Im}{\psi^{\sharp}} \cap \mathrm{Poiss}(S,\psi).
    \end{equation*}
It follows that
    \begin{equation*}
        \mathrm{Poiss}_{\mathrm{cot}}(S,\psi) = \{\psi^{\sharp}\theta \mid \dd{\theta}(\psi^{\sharp}\alpha, \psi^{\sharp}\beta) = 0,\ \text{for all}\ \alpha, \beta \in \Gamma\T^{\ast}S\}.
    \end{equation*}    
    
Consider the Lie subalgebras of Poisson vector fields defined by
    \begin{align*}
        &\mathrm{Poiss}'(S,\psi) := \{v \in \mathrm{Poiss}(S,\psi) \mid [v,w] \in \operatorname{Im}{\psi^{\sharp}},\ \text{for all}\ w \in \mathrm{Poiss}(S,\psi)\}, \\
        &\mathrm{Poiss}_{\mathrm{cot}}'(S,\psi) := \{\psi^{\sharp}\theta \mid \dd{\theta}(\psi^{\sharp}\alpha, \cdot ) = 0,\ \text{for all}\ \alpha \in \Gamma\T^{\ast}S\}, \\
        &\mathrm{Poiss}_{\mathrm{cot}}''(S,\psi) := \{\psi^{\sharp}\theta \mid \dd{\theta} = 0,\}.
    \end{align*}
Note that we have the inclusions
    \begin{equation*}
        \mathrm{Ham}(S,\psi) \,\subseteq\, \mathrm{Poiss}_{\mathrm{cot}}''(S,\psi) \,\subseteq\, \mathrm{Poiss}_{\mathrm{cot}}'(S,\psi) \,\subseteq\, \mathrm{Poiss}_{\mathrm{cot}}(S,\psi) \,\subseteq\, \mathrm{Poiss}'(S,\psi) \,\subseteq\, \mathrm{Poiss}(S,\psi).
    \end{equation*}
Consequently, the following Lie algebras are obstructions to the vanishing of the first Poisson cohomology of \ec{(S,\psi)}:
    \begin{equation}\label{eq:CohoLieSub}
        \frac{\mathrm{Poiss}_{\mathrm{cot}}''(S,\psi)}{\mathrm{Ham}(S,\psi)} \,\subseteq\, \mathrm{H}^{1}_{\mathrm{cot}}(S,\psi)' := \frac{\mathrm{Poiss}_{\mathrm{cot}}'(S,\psi)}{\mathrm{Ham}(S,\psi)} \,\subseteq\, \mathrm{H}^{1}_{\mathrm{cot}}(S,\psi) := \frac{\mathrm{Poiss}_{\mathrm{cot}}(S,\psi)}{\mathrm{Ham}(S,\psi)} \\ 
        \subseteq\, \frac{\mathrm{Poiss}'(S,\psi)}{\mathrm{Ham}(S,\psi)}.
    \end{equation}
Note that \ec{\mathrm{Poiss}_{\mathrm{cot}}''(S,\psi) / \mathrm{Ham}(S,\psi)} is just the image of the first de Rham cohomology of $S$ under the natural mapping in cohomology induced by $\psi^{\sharp}$, called the \emph{Lichnerowicz homomorphism} \cite{Va94}. Moreover, if $\mathcal{F}$ is regular, then \ec{\mathrm{H}^{1}_{\mathrm{cot}}(S,\psi)} is just the \emph{tangential first Poisson cohomology} of \ec{(S,\psi)} \cite{Li82}. So, in this case, we have that
    \begin{equation}\label{eq:TangCohom}
        \mathrm{H}^{1}_{\mathrm{cot}}(S,\psi) \simeq \mathrm{H}^{1}_{\mathcal{F}}(S), 
    \end{equation}
where \ec{\mathrm{H}^{1}_{\mathcal{F}}(S)} is the leafwise de Rham cohomology of $S$ \cite{Game2002}. In particular, if \ec{(S,\psi)} is symplectic, then the four quotients in \eqref{eq:CohoLieSub} coincide. In the general case, we call \ec{\mathrm{H}^{1}_{\mathrm{cot}}(S,\psi)} the \emph{cotangential first Poisson cohomology} of \ec{(S,\psi)}.

\begin{remark}
If we call a multivector field on $S$ \emph{cotangent} when it belongs to the image of the morphism \ec{\wedge^{\bullet}\psi^{\sharp}: \wedge^{\bullet}\T^{\ast}S \to \wedge^{\bullet}\T{S}}, then they give rise to a cochain subcomplex of \ec{(\Gamma\wedge^{\bullet}\T{S},\dd_{\psi})} whose cohomology can be called the \emph{cotangential Poisson cohomology} of \ec{(S,\psi)}, which in degree one agrees with \ec{\mathrm{H}^{1}_{\mathrm{cot}}(S,\psi)}.
\end{remark}

\paragraph{The Trivial Extension Algebra \ec{\Cinf{\mathrm{aff}}(E)}.}
The $\R{}$-algebra \ec{\Cinf{\mathrm{aff}}(E)} of fiberwise affine functions on a vector bundle \ec{p:E \to S} (in particular, on a normal bundle) fits in the following short exact sequence of \ec{\Cinf{S}}-algebras
    \begin{equation*}
        0 \to \Gamma{E^{\ast}} \longhookrightarrow \Cinf{\mathrm{aff}}(E) \longtwoheadrightarrow \Cinf{S} \to 0,
    \end{equation*}
where the product on \ec{\Gamma{E^{\ast}}} is trivial. The pull-back \ec{p^*:\Cinf{S}\to\Cinf{\mathrm{aff}}(E)} leads to the natural identification \ec{\Cinf{\mathrm{aff}}(E) \simeq \Cinf{S} \oplus \Gamma{E^{\ast}}},
under which the \ec{\R{}}-algebra structure is given by
    \begin{equation}\label{eq:ProductTEA}
        (f\oplus\eta) \cdot (g\oplus\xi) := fg\oplus(f\xi+g\eta).
    \end{equation}

In order to characterize the derivations of \ec{\Cinf{\mathrm{aff}}}, we recall that a \emph{derivative endomorphism} \cite{KosMac2002} of \ec{\Gamma{E^{\ast}}} is an $\R$-linear mapping \ec{\delta:\Gamma{E^{\ast}} \to \Gamma{E^{\ast}}} such that there exists a (unique) \ec{u \in \X{S}}, called the symbol of $\delta$, satisfying
    \begin{equation*}
        \delta(f\eta) = f\,\delta(\eta) + (\dlie{u}{f})\eta, \quad f \in \Cinf{S}, \eta \in \Gamma{E^{\ast}}.
    \end{equation*}
The $\Cinf{S}$-module \ec{\mathbb{D}(\Gamma{E^{\ast}})} of all derivative endomorphisms of \ec{\Gamma{E^{\ast}}} is an $\R$-Lie algebra with the commutator. Moreover,  the $\Cinf{S}$-linear mapping
    \begin{equation}\label{eq:sigma}
        \sigma:\delta\longmapsto \sigma_{\delta} := u
    \end{equation}
is a surjective $\R{}$-Lie algebra morphism. Indeed, every covariant derivative $\nabla$ on $E^{\ast}$ induces a right inverse \ec{\X{S} \ni u \to \nabla_{u} \in \mathbb{D}(\Gamma E^{\ast})}: \ec{\sigma_{\nabla_{u}} = u}. In consequence, the module of derivative endomorphisms of \ec{\Gamma{E^{\ast}}} fit in the following short exact sequence of $\R{}$-Lie algebras:
    \begin{equation}\label{eq:Short1}
        0 \to \Gamma(\mathrm{End}(E^\ast))\longhookrightarrow \mathbb{D}(\Gamma{E^{\ast}}) \overset{\sigma}{\longtwoheadrightarrow} \Gamma \T S \to 0.
    \end{equation}

\begin{lemma}\label{lema:Der1}
Every $\R{}$-linear derivation $X$ of \ec{\Cinf{\mathrm{aff}}(E)} is of the form \ec{X = X_{Q} + X_{\delta}}, where
    \begin{equation*}
        X_{Q}(f \oplus \eta) = 0 \oplus Q(f) \qquad \text{and} \qquad X_{\delta}(f \oplus \eta) = \dlie{\sigma_{\delta}}f \oplus \delta(\eta),
    \end{equation*}
for uniques \ec{Q \in \Gamma(\T S\otimes E^{\ast})} and  \ec{\delta \in \mathbb{D}(\Gamma{E^{\ast}})}, with \ec{f \oplus \eta \in \Cinf{\mathrm{aff}}(E)}. In particular, we have a short exact sequence 
    \begin{equation}\label{eq:Short2}
        0 \to \Gamma(\T S\otimes E^\ast)\overset{j}{\longhookrightarrow} \mathrm{Der}_{\R{}}(\Cinf{\mathrm{aff}}(E)) \longtwoheadrightarrow \mathbb{D}(\Gamma{E^{\ast}}) \to 0,
    \end{equation}
with \ec{j(Q) := X_{Q}}.
\end{lemma}

\begin{proof}
It is easy to see that $X_{Q}$ and $X_{\delta}$ are derivations of \ec{\Cinf{\mathrm{aff}}(E)}. Reciprocally,  given a derivation $X$ of $\Cinf{\mathrm{aff}}(E)$, the key point is that
    \begin{equation}\label{eq:IsXZero}
        \iota_{S}^{\ast} \circ X |_{\Gamma{E^{\ast}}} = 0,
    \end{equation}
with \ec{\iota_{S}: S \hookrightarrow E} the zero section. Indeed, 
we have that \ec{\iota_{S}^{\ast}\circ X|_{\Gamma{E^{\ast}}} = \langle s,\cdot \rangle}, for some \ec{s \in \Gamma{E}} such that \ec{\langle s,\eta \rangle \eta=0} for all \ec{\eta \in \Gamma{E^{\ast}}}, implying that \ec{s=0}. Moreover, we have that \ec{X|_{\Cinf{S}}(f \oplus 0) = \dlie{u}f \oplus Q(f)}, for some \ec{u \in \X{S}}, and \ec{Q \in \Gamma(\T{S} \otimes E^{\ast})} and \ec{X|_{\Gamma{E^{\ast}}} = X_\delta}, for some \ec{\delta \in \mathbb{D}(\Gamma{E^{\ast}})} satisfying \ec{\sigma_{\delta}=u}. 
\end{proof}

In other words, we have a parametrization of the derivations of \ec{\Cinf{\mathrm{aff}}(E)} by pairs consisting of a $\Gamma{E^{\ast}}$-valued vector field on $S$ and a derivative endomorphism of $\Gamma{E^{\ast}}$. We use this characterization to derive the main results of this work.

Given \ec{u \in \Gamma\T{S}}, we say that a derivation \ec{X=X_{Q} + X_{\delta}} is an \emph{extention} of $u$ to \ec{\Cinf{\mathrm{aff}}(E)} if \ec{\sigma_{\delta}=u}. Clearly, such an extension is not unique.

\begin{remark}
In the more general setting of derivations of trivial extension algebras, unlike the geometric case, the term corresponding to the left-hand side of \eqref{eq:IsXZero} does not necessarily vanish \cite{GaRuVo23} (see also \cite{Zhu2020}).
\end{remark}

    \section{Setting of the Problem and Main Results} \label{sec:MainResults}

Let \ec{(S,\psi)} be an embedded Poisson submanifold of a Poisson manifold \ec{(M,\pi)}. Recall from \cite{RuGaVo20} (see also \cite{GaRuVo23}) that there exists a Poisson algebra that provides a first-order approximation of \ec{(\Cinf{M}, \{\cdot,\cdot\}_{\pi})} around $S$, in the following sense: \emph{given an exponential map \ec{\mathbf{e}:E \rightarrow M}, there exists a Poisson algebra
    \begin{equation*}
        \PP=(\Cinf{\mathrm{aff}}(E),\{\cdot,\cdot\}^{\mathrm{aff}})      
    \end{equation*}
such that
    \begin{equation*}
        \{\phi_{1} \circ \mathbf{e}^{-1},\phi_{2} \circ \mathbf{e}^{-1}\}_{\pi} \circ \mathbf{e} \,=\, \{\phi_{1},\phi_{2}\}^{\mathrm{aff}} + \OO(2), \quad \phi_{1},\phi_{2} \in \Cinf{\mathrm{aff}}(E);
    \end{equation*}
around the zero section \ec{S \hookrightarrow E}}.
Here, $E$ is the normal bundle of $S$ and \ec{\Cinf{\mathrm{aff}}(E)} is the algebra of fiberwise affine functions on $E$ with product \eqref{eq:ProductTEA}. Furthermore, different exponential maps give rise to isomorphic Poisson algebras. Thus, the approximating Poisson algebra is termed the \emph{infinitesimal Poisson algebra of $S$}.

We address the following problems concerning the first cohomology \ec{\mathrm{H}^{1}(\PP)} of the infinitesimal Poisson algebra of $S$:
    \begin{itemize}
        \item Express \ec{\mathrm{H}^{1}(\PP)} in terms of the cohomology of \ec{(S,\psi)} and the Lie algebra cohomology of $E^{\ast}$.
        \item Determine whether there exists a relation between \ec{\mathrm{H}^{1}(\PP)} and the first Poisson cohomology of \ec{(S,\psi)}.
        \item Provide necessary and sufficient conditions for the vanishing of \ec{\mathrm{H}^{1}(\PP)}.
    \end{itemize}
To this end, we first recall some notions and facts. 

It is known that the co-normal bundle \ec{E^{\ast} \simeq \T{S^\circ}} over \ec{(S,\psi)} is a bundle of Lie algebras not necessarily locally trivial (see, for example, \cite{CraFeMa21}).
Indeed, we have a \ec{\Cinf{S}}-Lie algebra
    \begin{equation}\label{eq:GLieAlg}
        \mathscr{G}:=(\Gamma E^{\ast}, [\cdot,\cdot]_1)
    \end{equation}
characterized by \ec{[\mathrm{d}f|_S,\mathrm{d}g|_S]_1=\mathrm{d}\{f,g\}_{\pi}|_S} for all \ec{f,g\in C^{\infty}_{M}} vanishing on $S$. 

The Lie algebra \eqref{eq:GLieAlg} and the choice of an exponential map \ec{\mathbf{e}:E \rightarrow M} allow us to write the Poisson bracket of $\PP$ as follows \cite{RuGaVo20} (see also \cite[Chapter 5]{Mar13}):
    \begin{equation}\label{eq:Bracket_PT}
        \{f \oplus \eta, g \oplus \xi\}^{\mathrm{aff}} := \psi(\dd{f},\dd{g}) \oplus \big( \D_{\dd{f}}\xi - \D_{\dd{g}} \eta + [\eta, \xi]_{1} + \K(\dd{f},\dd{g}) \big), \quad f\oplus\eta, g\oplus\xi \in \Cinf{\mathrm{aff}}(E);
    \end{equation}
where the $\mathbf{e}$-depending data \ec{(\D,\K)} consist of (see \cite[Section 4]{RuGaVo20})
    \begin{itemize}
      \item a contravariant derivative \ec{\D: \Gamma\T^{\ast}S \times \Gamma E^* \longrightarrow \Gamma E^*} on $E^{\ast}$; and

      \item a \ec{E^*}-valued bivector field \ec{\K \in \Gamma(\wedge^{2}\T{S} \otimes E^*)};
    \end{itemize}
and satisfy conditions \eqref{eq:PT1}-\eqref{eq:PT3}. Moreover, we can define a cochain complex
    \begin{equation}\label{eq:Cochain}
        \big( \mathfrak{Z}^{\bullet},\partial_{\D} \big),
    \end{equation}
where \ec{\mathfrak{Z}^{\bullet} := \oplus_{k \in \mathbb{Z}}\,\Gamma (\wedge^{k}\T{S}\otimes Z_{\mathscr{G}})} is the graded \ec{\Gamma(\wedge^{\bullet}\T{S})}-module  of multivector fields on $S$ with values on the center \ec{Z_{\mathscr{G}}} of \ec{\mathscr{G}}, and the coboundary operator $\partial_{\D}$ is the restriction to \ec{\mathfrak{Z}^{\bullet}} of the contravariant differential \ec{\dd_\D} in \eqref{eq:dQ} induced by $\D$. Moreover, by formula \eqref{eq:Bracket_PT}, the mapping $j$ in \eqref{eq:Short2} induces an $\R{}$-linear mapping in cohomology,
    \begin{equation}\label{eq:J}
        \mathrm{J}:\mathrm{H}_{\partial_{\D}}^{1} \longrightarrow \mathrm{H}^{1}(\PP),
    \end{equation}
where \ec{\mathrm{H}_{\partial_{\D}}^{1}} is the first cohomology of the cochain complex \ec{\big( \mathfrak{Z}^{\bullet},\partial_{\D} \big)}.

We now proceed to formulate the main results of this work.

\paragraph{First Poisson Cohomology.}
We present a description of the first cohomology \ec{\mathrm{H}^{1}(\PP)} of the infinitesimal Poisson algebra of $S$. To this end, we consider the following objects:
\begin{itemize}
    \item \ec{\mathrm{Ham}(S,\psi)}, the Lie algebra of Hamiltonian vector fields of $S$;
    \item \ec{\operatorname{Inn}{\mathscr{G}}}, the Lie ideal of inner derivations of the Lie algebra \ec{\mathscr{G}} in \eqref{eq:GLieAlg}; and
    \item \ec{\mathfrak{M}(\PP),\,\mathscr{C}(\PP),\,\mathfrak{M}_{0}(\PP)} and \ec{\mathscr{C}_{0}(\PP)}, the submodules of \ec{\mathbb{D}(\Gamma{E^{\ast}})} defined in \eqref{eq:M}, \eqref{eq:C}, \eqref{eq:M0} and \eqref{eq:C0}, respectively;
\end{itemize}

    \newpage

\begin{theorem}\label{teo:ShortExactSeq}
The first cohomology of the infinitesimal Poisson algebra of a Poisson submanifold \ec{(S,\psi)} fits in the following diagram of short exact sequences:
    \begin{equation*}
        \xymatrix{
            & & & 0 \ar[d] \\
            & & & \displaystyle\frac{\mathfrak{M}_{0}(\PP)}{\mathscr{C}_{0}(\PP) + \operatorname{Inn}{\mathscr{G}}}\ar@{^{(}->}[d] \\
            0 \ar[r] & \displaystyle\frac{\mathrm{H}_{\partial_{\D}}^{1}}{\ker{\mathrm{J}}} \,\ar@{^{(}->}[r]
                & \mathrm{H}^{1}(\PP) \ar@{->>}[r] & \displaystyle\frac{\mathfrak{M}(\PP)}{\mathscr{C}(\PP)+\operatorname{Inn}{\mathscr{G}}} \ar[r] \ar@{->>}[d] & 0 \\
            & & & \displaystyle\frac{\operatorname{Im}{\sigma|_{\mathfrak{M}(\PP)}}}{\mathrm{Ham}(S,\psi)} \ar[d] \\
            & & & 0 }
    \end{equation*}
Here, $\sigma$ and $\mathrm{J}$ are the mappings given in \eqref{eq:sigma} and \eqref{eq:J}, respectively.
\end{theorem}
The \nameref{para:TeoShortExactSeq} is given in Section \ref{sec:proof}. Intuitively, the horizontal sequence of the diagram is induced by the short exact sequence in \eqref{eq:Short2}, and the vertical one by the  short exact sequence in \eqref{eq:Short1}. In particular, Theorem \ref{teo:ShortExactSeq} says that \ec{\mathrm{H}^{1}(\PP)} can be described in terms of 
equivalence classes of vector fields on $S$ with values in the center of $\mathscr{G}$, derivative endomorphisms of \ec{\Gamma{E^{\ast}}} and Poisson vector fields of $\psi$.  More precisely, the Lie algebras in the diagram are related with some intrinsic cohomologies of \ec{(S,\psi)}:
\begin{itemize}
    \item \ec{\mathrm{H}_{\partial_{\D}}^{1}} is the first cohomology of the cochain complex in \eqref{eq:Cochain} associated with the center of the Lie algebra $\mathscr{G}$ in  \eqref{eq:GLieAlg}.
    \item \ec{\mathfrak{M}_{0}(\PP) / (\mathscr{C}_{0}(\PP) + \operatorname{Inn}{\mathscr{G}})} admits an inclusion to the first Lie algebra cohomology of \ec{\mathscr{G}} with coefficients in the adjoint representation, and
    \item the Lie subalgebra
        \begin{equation}\label{eq:Inclusion}
            \frac{\operatorname{Im}{\sigma|_{\mathfrak{M}(\PP)}}}{\mathrm{Ham}(S,\psi)} \subseteq \mathrm{H}^1(S,\psi)
        \end{equation}
    of the first Poisson cohomology of \ec{(S,\psi)} that consists of the cohomology clases of Poisson vector fields of $\psi$ which, as derivations of the algebra $\Cinf{S}$, can be extended to a Poisson derivation of $\PP$.
\end{itemize}

\paragraph{The Mapping \ec{\boldsymbol{\mathrm{H}^{1}(\PP) \to \mathrm{H}^1(S,\psi)}}.}
As a consequence of Theorem \ref{teo:ShortExactSeq}, the first cohomology of the infinitesimal Poisson algebra of $S$ and its first Poisson cohomology are related. More precisely, taking into account inclusion \eqref{eq:Inclusion}, there exists an $\R{}$-linear mapping
    \begin{equation}\label{eq:ZetaMap}
        \zeta: \mathrm{H}^{1}(\PP) \to \mathrm{H}^1(S,\psi),
    \end{equation}
given by the following composition:
    \begin{equation*}
        \mathrm{H}^{1}(\PP) \longtwoheadrightarrow \frac{\mathfrak{M}(\PP)}{\mathscr{C}(\PP) + \operatorname{Inn}{\mathscr{G}}} \longtwoheadrightarrow \frac{\operatorname{Im}{\sigma|_{\mathfrak{M}(\PP)}}}{\mathrm{Ham}(S,\psi)} \longhookrightarrow \mathrm{H}^1(S,\psi)
    \end{equation*}
The mapping $\zeta$ is not surjective in general, and the question of whether it is or not is non-trivial (see, for example, \cite{Zhu2020,GaRuVo23}). Theorems \ref{teo:SurjSym} and \ref{teo:PartiallySplit} provide a positive answer to this question in some cases.

Clearly, \ec{\operatorname{Im}{\zeta} = \operatorname{Im}{\sigma|_{\mathfrak{M}(\PP)}} / \mathrm{Ham}(S,\psi)}. So, in view of the interpretation of the Lie subalgebra in \eqref{eq:Inclusion} given above, it is important to know some properties of the image of $\zeta$ in order to understand the first cohomology of $\PP$.

\begin{theorem}\label{teo:H1primprim}
For a Poisson submanifold \ec{(S,\psi)}, the image of the mapping \eqref{eq:ZetaMap} contains the following Lie subalgebras of \ec{\mathrm{H}^{1}(S,\psi)}:
\begin{enumerate}[label=(\textrm{\alph*}) ]
    \item the image under the Lichnerowicz homomorphism \ec{\mathrm{H}^{1}_{\mathrm{dR}}(S)\to\mathrm{H}^{1}(S,\psi)} induced by $\psi$ of the first de Rham cohomology of $S$; \label{item:TeoImag1}
    
    \item the quotient \ec{\mathrm{H}^{1}_{\mathrm{cot}}(S,\psi)'} defined in \eqref{eq:CohoLieSub}; and \label{item:TeoImag2}
    
    \item the cohomology classes of restrictions to \ec{(S,\psi)} of Poisson vector fields of \ec{(M,\pi)} tangent to $S$: \label{item:cH1primprim}
        \begin{equation}\label{eq:PoissRestrict}
            \frac{\{v \in \mathrm{Poiss}(S,\psi) \mid v = Z|_{S},\, Z \in \mathrm{Poiss}(M,\pi)\}}{\mathrm{Ham}(S,\psi)}.
        \end{equation}
\end{enumerate}
Moreover, if \ec{(S,\psi)} satisfies the condition
    \begin{equation}\label{eq:CasimIn}
        \mathrm{Casim}(S,\psi) \oplus \{0\} \subseteq \mathrm{Casim}(\PP),
    \end{equation}
then the image of the mapping \eqref{eq:ZetaMap} also contains the cotangential first Poisson cohomology \ec{\mathrm{H}^{1}_{\mathrm{cot}}(S,\psi)} defined in \eqref{eq:CohoLieSub}.
\end{theorem}
The \nameref{para:TeoH1primprim} is given in Section \ref{sec:proof}. In particular, item \ref{item:cH1primprim} relies on the fact that every Poisson vector field $Z$ of $\pi$ tangent to $S$ induces a derivation of $\PP$ which decends to the Poisson vector field of $\psi$ given by $Z|_{S}$ \cite{RuGaVo20,GaRuVo23}.

Note that in Theorem \ref{teo:H1primprim}, the Lie algebra in item \ref{item:TeoImag2} contains the Lie algebra in item \ref{item:TeoImag1}, by the inclusions in \eqref{eq:CohoLieSub}. Moreover, the Lie algebra in  \eqref{eq:PoissRestrict} also contains the Lie algebra in item \ref{item:TeoImag1}, since every Poisson vector field \ec{\psi^{\sharp}(\theta)} on \ec{(S,\psi)}, with \ec{\dd\theta = 0}, induces a Poisson vector field $Z$ of \ec{(M,\pi)} that is tangent to $S$ and restrict to \ec{\psi^{\sharp}(\theta)}.

To get more insight on condition \eqref{eq:CasimIn}, we note that every Casimir element of $\PP$ induces a Casimir function of \ec{(S,\psi)}, but the converse is not true in general (see Lemma \ref{lema:CasimP}). In particular, condition \eqref{eq:CasimIn} holds if $S$ only admits locally constant Casimir functions. This occurs, for example, if the characteristic foliation of $S$ is open-book \cite{Ito2014-1,Ito2014-2,Licata24}, or if $S$ is symplectic almost everywhere, such as in the case when $S$ is symplectic or log-symplectic ($b$-symplectic) \cite{Nest1996,Guill2014}. Moreover, as Lemma \ref{lema:CasimP} states, condition \eqref{eq:CasimIn} is equivalent to the fact that
    \begin{equation*}
        \mathrm{Casim}(\PP) \simeq \mathrm{Casim}(S,\psi) \oplus \mathrm{H}^{0}_{\partial_{\D}},
    \end{equation*}
where \ec{\mathrm{H}^{0}_{\partial_{\D}}} is the zeroth cohomology of the cochain complex \ec{\big( \mathfrak{Z}^{\bullet},\partial_{\D} \big)} in \eqref{eq:Cochain}.

For regular Poisson manifolds, Theorem \ref{teo:H1primprim} gives the following:

\begin{corollary}\label{cor:CasimInRegular}
If \ec{(S,\psi)} is a regular Poisson submanifold for which property \eqref{eq:CasimIn} holds, then the image of the mapping \eqref{eq:ZetaMap} contains the tangential first Poisson cohomology of \ec{(S,\psi)}.
\end{corollary}

Due to the isomorphism \eqref{eq:TangCohom}, for regular Poisson submanifolds with property \eqref{eq:CasimIn} the geometric and dynamical features of the characteristic foliation plays a important role in the analysis of the first Poisson cohomology of the infinitesimal Poisson algebra. See, for instance, Corollary \ref{teo:VanishCasimInReg}.

\paragraph{Vanishing of the First Cohomology: Sufficiency Criteria.} We present some sufficient conditions for the vanishing of the first cohomology of the infinitesimal Poisson algebra of $S$. Our first result involves three types of intrinsic cohomologies associated with a Poisson submanifold.

\begin{theorem}\label{teo:gen_main}
The first cohomology of the infinitesimal Poisson algebra of a Poisson submanifold \ec{(S,\psi)} vanishes if the following cohomologies in degree one are trivial:
    \begin{enumerate}[label=(\alph*)]
        \item The Poisson cohomology of $(S,\psi)$.\label{cond2:H1(P0)trivial}
        \item The Lie algebra cohomology of \ec{\mathscr{G}} in \eqref{eq:GLieAlg} with coefficients in the adjoint representation.\label{cond2:inner}
        \item The cohomology of the cochain complex \ec{\big( \mathfrak{Z}^{\bullet},\partial_{\D} \big)} in \eqref{eq:Cochain}.\label{cond2:centerless}
    \end{enumerate}
\end{theorem}
The \nameref{para:Teogen_main} is given in Section \ref{sec:proof}.

As we show in the next paragraphs, item \ref{cond2:centerless} in this theorem is a necessary condition for the case of Poisson submanifolds with property \eqref{eq:CasimIn}. Moreover, items \ref{cond2:H1(P0)trivial} and \ref{cond2:centerless} are necessary conditions if the Poisson submanifold is symplectic.

An immediate consequence of Theorem \ref{teo:gen_main} is the following:

\begin{corollary}\label{cor:main}
The first cohomology of the infinitesimal Poisson algebra of a Poisson submanifold \ec{(S,\psi)} vanishes if
\begin{enumerate}[label=(\alph*)]
    \item the first Poisson cohomology of \ec{(S,\psi)} is trivial;  \label{cond:H1(P0)trivial}
    \item every \emph{\ec{\Cinf{S}}-linear} derivation of \ec{\mathscr{G}} is inner; and \label{cond:inner}
    \item the Lie algebra \ec{\mathscr{G}} is centerless. \label{cond:centerless}
\end{enumerate}
\end{corollary}
We remark that conditions in this corollary are natural but not self-evident since $S$ is a \emph{singular} Poisson submanifold, in general. In particular, we formulate the following criteria:

\begin{proposition}\label{prop:loctriv}
Suppose that the conormal bundle of a Poisson submanifold \ec{(S,\psi)} is a locally trivial Lie bundle. Then, the first cohomology of the infinitesimal Poisson algebra of $S$ vanishes if
    \begin{itemize}
        \item the first Poisson cohomology of \ec{(S,\psi)} is trivial; and
        \item the typical fiber is semisimple.
    \end{itemize}
\end{proposition}

This proposition is consequence of a partition of unity argument and the following known fact: the semisimple assumption  for the typical fiber implies that conditions \ref{cond:inner} and \ref{cond:centerless} in Corollary \ref{cor:main} hold fiberwise (see \cite[Whitehead's Lemma]{Jacob62}).

\paragraph{Vanishing of the First Cohomology: Necessity Criteria.} Now, we present some necessary conditions for the vanishing of \ec{\mathrm{H}^{1}(\PP)}.

\begin{theorem}\label{teo:VanishLich}
If the first cohomology of the infinitesimal Poisson algebra of a Poisson submanifold \ec{(S,\psi)} of \ec{(M,\pi)} vanishes, then
\begin{itemize}
    \item the Lichnerowicz homomorphism \ec{\mathrm{H}^{1}_{\mathrm{dR}}(S)\to\mathrm{H}^{1}(S,\psi)} induced by $\psi$ is the zero mapping; and
    \item every Poisson vector field of \ec{(M,\pi)} tangent to $S$ restricts to a Hamiltonian vector field of $\psi$.
\end{itemize}
Moreover, if $S$ satisfies condition \eqref{eq:CasimIn}, then are trivial:
\begin{itemize}
    \item the cotangential first Poisson cohomology \ec{\mathrm{H}^{1}_{\mathrm{cot}}(S,\psi)} defined in \eqref{eq:CohoLieSub}; and
    \item the first cohomology of the cochain complex \ec{\big( \mathfrak{Z}^{\bullet},\partial_{\D} \big)} in \eqref{eq:Cochain}.
\end{itemize}
\end{theorem}
\begin{proof}
The first three items follow from Theorem \ref{teo:H1primprim} since \ec{\mathrm{H}^{1}(\PP)=0} implies that the image of the mapping in \eqref{eq:ZetaMap} is zero. The last one follows from Theorem \ref{teo:ShortExactSeq} since, by Lemma \ref{lema:CasimP}, condition \eqref{eq:CasimIn} implies that \ec{\ker{\mathrm{J}}=0}.
\end{proof}

As mentioned in \cite{Duf05}, an interesting and largely open question is: under what conditions is the Lichnerowicz homomorphism injective or surjective? So, Theorem \ref{teo:VanishLich} gives a negative criterion for Poisson submanifolds. 

Finally, by Corollary \ref{cor:CasimInRegular} and isomorphism \eqref{eq:TangCohom}, we have the following:

\begin{corollary}\label{teo:VanishCasimInReg}
Let \ec{(S,\psi)} be a regular Poisson submanifold satisfying property \eqref{eq:CasimIn}. If the first cohomology of the infinitesimal Poisson algebra of $S$ vanishes, then the following cohomologies in degree one are trivial:
    \begin{itemize}
        \item the leafwise de Rham cohomology of the characteristic foliation of $S$; and
        \item the cohomology of the cochain complex \ec{\big( \mathfrak{Z}^{\bullet},\partial_{\D} \big)} in \eqref{eq:Cochain}.
    \end{itemize}
\end{corollary}

So, regular Poisson submanifolds with property \eqref{eq:CasimIn} and which have infinitesimal Poisson algebra with trivial first cohomology have \emph{generically} simply connected symplectic leaves.

    \subsection{The Symplectic Leaf Case}

As an application of the previous results, we present a study of the first cohomology of the infinitesimal Poisson algebra of a symplectic leaf of the Poisson manifold \ec{(M,\pi)}.

First, we formulate the following positive answer to the surjectivity problem of the mapping \eqref{eq:ZetaMap}:

\begin{theorem}\label{teo:SurjSym}
If \ec{(S,\psi)} is a symplectic leaf of \ec{(M,\pi )}, then the mapping \ec{\mathrm{H}^{1}(\PP)\to \mathrm{H}^1(S,\psi)} in \eqref{eq:ZetaMap} is surjective. Moreover, the first cohomology of the infinitesimal Poisson algebra of $S$ splits as follows:
    \begin{equation*}
        \mathrm{H}^{1}(\PP) \simeq \mathrm{H}^{1}_{\mathrm{dR}}(S) \oplus \mathrm{H}_{\partial_{\D}}^{1} \oplus \frac{\mathfrak{M}_{0}(\PP)}{\operatorname{Inn}{\mathscr{G}}}.
    \end{equation*}
Here, \ec{\mathrm{H}^{1}_{\mathrm{dR}}(S)} is the first de Rham cohomology of $S$.
\end{theorem}
\begin{proof}
Since $S$ is symplectic the surjectivity follows from item \ref{item:TeoImag1} of Theorem \ref{teo:H1primprim}, which implies that
    \begin{equation}\label{eq:ImH1H1dR}
        \frac{\operatorname{Im}{\sigma|_{\mathfrak{M}(\PP)}}}{\mathrm{Ham}(S,\psi)} \simeq \mathrm{H}^{1}(S,\psi) \simeq \mathrm{H}^{1}_{\mathrm{dR}}(S).
    \end{equation}
On the other hand, by Lemma \ref{lema:CasimP}, we have that \ec{\ker{\mathrm{J}}=0} and \ec{\mathscr{C}_{0}(\PP)=\{0\}} since condition \eqref{eq:CasimIn} automatically holds. Hence, by Theorem \ref{teo:ShortExactSeq}, the splitting follows.
\end{proof}

As mentioned above, the first cohomology \ec{\mathrm{H}_{\partial_{\D}}^{1}} of the cochain complex in \eqref{eq:Cochain} is associated to the center of the Lie algebra $\mathscr{G}$ in \eqref{eq:GLieAlg}, and the quotient \ec{\mathfrak{M}_{0}(\PP) / \operatorname{Inn}{\mathscr{G}}} admits an inclusion to the first Lie algebra cohomology of \ec{\mathscr{G}}. 

Taking into account Theorems \ref{teo:VanishLich} and \ref{teo:SurjSym}, we have the following necessary conditions for the vanishing of \ec{\mathrm{H}^{1}(\PP)}.
    
\begin{corollary}\label{cor:SymplecticLeaf}
If the first cohomology of the infinitesimal Poisson algebra of a symplectic leaf \ec{(S,\psi)} of \ec{(M,\pi)} vanishes, then
    \begin{itemize}
        \item the de Rham cohomology of $S$ and the cohomology of the cochain complex \ec{\big( \mathfrak{Z}^{\bullet},\partial_{\D} \big)} in \eqref{eq:Cochain} are trivial in degree one; and
        \item every Poisson vector field of $\pi$ tangent to $S$ restricts to a Hamiltonian vector field of $\psi$.
    \end{itemize}
\end{corollary}

So, symplectic leaves which have infinitesimal Poisson algebra with trivial first cohomology are \emph{generically} simply connected. With this in mind, we formulate some sufficient conditions for the vanishing of \ec{\mathrm{H}^{1}(\PP)}. To this end, we recall that the co-normal bundle of a symplectic leaf is a locally trivial Lie bundle \cite{Duf05,Mac95}, and the typical fiber is called its \emph{isotropy} Lie algebra \cite{CraFeMa21}.

\begin{proposition}\label{prop:SymplecticVanish}
For simply connected symplectic leaves with semisimple isotropy Lie algebra,
the first cohomology of its infinitesimal Poisson algebra vanishes.
\end{proposition}

This result is a direct consequence of Proposition \ref{prop:loctriv} and isomorphism \eqref{eq:ImH1H1dR}.

    \subsection{The Partially Split Case}

Here, we describe the first cohomology of the infinitesimal Poisson algebra of Poisson submanifolds with partially split first-order jet \cite{FerMar22}. In particular, we establish cohomological obstructions to the partially split property.


As well as the infinitesimal Poisson algebra approach, a characterization of the first-order jet of the Poisson submanifold \ec{(S,\psi)\hookrightarrow(M,\pi)} is provided by the restricted Lie algebroid 
    \begin{equation}\label{eq:AlgbdRestPsplit}
        \big( \T^{\ast}_{S}M, \pi^{\sharp}|_{S}, [\cdot,\cdot]_{S} \big),
    \end{equation}
which arises as the restriction of the cotangent Lie algebroid of \ec{(M,\pi)} to $S$ \cite{Mar12}. Moreover, the restricted Lie algebroid fits in the short exact sequence of Lie algebroids
    \begin{equation*}
        0 \to E^{\ast} \simeq \T^{\circ}S \longhookrightarrow \T^{\ast}_{S}M \overset{r}{\longtwoheadrightarrow} \T{S}^{\circ} \to 0,
    \end{equation*}
where \ec{r:\T^{\ast}_{S}M \to \T^{\ast}S} is the restriction mapping of 1-forms.

    \newpage

We recall from \cite{FerMar22} that the Lie algebroid in \eqref{eq:AlgbdRestPsplit} is said to be \emph{partially split} if it admits an infinitesimally multiplicative connection 1-form, that is, a pair \ec{(L,l)} consisting of an $\R{}$-linear mapping \ec{L:\Gamma\T^{\ast}_{S}M \to \Gamma(\T^\ast S \otimes E)} and a vector bundle mapping \ec{l:\T^{\ast}_{S}M \to E^\ast} such that \ec{l|_{E^\ast}=\mathrm{Id}_{E^\ast}} and satisfying conditions \eqref{eq:PartiallySplit1} (see \cite[Definition A6]{FerMar22}). This implies the existence of a covariant derivative \ec{\nabla:\Gamma\T S\times\Gamma E^\ast\to\Gamma E^\ast} on $E^{\ast}$ and a $E^{\ast}$-valued tensor field \ec{U\in\Gamma(\T S\otimes\T S^\ast\otimes E^\ast)} on $S$, called \emph{coupling data} (see \cite[Section 5.4]{FMEhresmann}).

The key point is that, associated with an infinitesimally multiplicative connection 1-form, there is parameterizing data \ec{(\D,\K)} of the infinitesimal Poisson algebra of $S$ satisfying (Lemma \ref{lemma:PartialSplitData})
    \begin{equation*}
        \D_\alpha = \nabla_{\psi^{\sharp}\alpha} \quad \text{and} \quad \K(\alpha,\beta)=U(\alpha,\psi^{\sharp}\beta),
    \end{equation*}
for all \ec{\alpha,\beta \in \Gamma\T^{\ast}S}. This observation allows us to give a positive answer to the surjectivity question of the mapping \eqref{eq:ZetaMap}.

\begin{theorem}\label{teo:PartiallySplit}
If a Poisson submanifold \ec{(S,\psi)} admits a partially split first-order jet, then the mapping in \eqref{eq:ZetaMap},
    \begin{equation*}
        \mathrm{H}^{1}(\PP) \to \mathrm{H}^1(S,\psi),
    \end{equation*}
is surjective. Moreover, the first cohomology of the infinitesimal Poisson algebra of $S$ admits the following splitting:
    \begin{equation}\label{eq:H1PartialSplit}
        \mathrm{H}^{1}(\PP) \simeq \mathrm{H}^1(S,\psi) \oplus \mathrm{H}_{\partial_{\D}}^{1} \oplus \frac{\mathfrak{M}_{0}(\PP)}{\operatorname{Inn}{\mathscr{G}}}.
    \end{equation}
Here, \ec{\mathrm{H}_{\partial_{\D}}^{1}} is the first cohomology of the cochain complex in \eqref{eq:Cochain} and \ec{{\mathfrak{M}_{0}(\PP)} / {\operatorname{Inn}{\mathscr{G}}}} is a Lie subalgebra of the first cohomology of the Lie algebra in \eqref{eq:GLieAlg} with coefficients in the adjoint representation, where the ideal \ec{\mathfrak{M}_{0}(\PP) \subseteq \mathbb{D}(\Gamma E^{\ast})} is defined in \eqref{eq:M0}.
\end{theorem}
The \nameref{para:TeoPartialSplit} is presented in Section \ref{sec:proof}. Moreover, we have the following necessary conditions for the partially split property.

\begin{proposition}
Let \ec{(S,\psi)} be a Poisson submanifold with partially split first-order jet. Then:
    \begin{enumerate}
        \item\ec{\mathrm{Casim}(S,\psi) \oplus \{0\} \subseteq \mathrm{Casim}(\PP)}.
        \item \ec{\ker{\mathrm{J}} = \{0\}}, where $\mathrm{J}$ is the mapping defined in \eqref{eq:J}.
        \item The coupling data \ec{(\nabla,U)} induces a right inverse \ec{\mathrm{H}^{1}(S,\psi)\hookrightarrow\mathrm{H}^1(\PP)} of the mapping in \eqref{eq:ZetaMap}.
    \end{enumerate}
\end{proposition}

So, taking into account Theorem \ref{teo:VanishLich}, we have the following:

\begin{corollary}
Let \ec{(S,\psi)} be a Poisson submanifold  admitting a partially split first-order jet. If the first cohomology of the infinitesimal Poisson algebra of $S$ vanishes, then
    \begin{itemize}
        \item the Poisson cohomology of \ec{(S,\psi)} and the cohomology of the cochain complex \ec{\big( \mathfrak{Z}^{\bullet},\partial_{\D} \big)} in \eqref{eq:Cochain} are trivial in degree one; and 
        \item every Poisson vector field of \ec{(M,\pi)} tangent to $S$ restricts to a Hamiltonian vector field of $\psi$.
    \end{itemize}
\end{corollary}

Conversely, the following conditions are sufficient for the vanishing of \ec{\mathrm{H}^{1}(\PP)}.

\begin{proposition}
Let \ec{(S,\psi)} be a Poisson submanifold  admitting a partially split first-order jet. Then, the first cohomology of the infinitesimal Poisson algebra of $S$ vanishes if
    \begin{itemize}
        \item the Poisson cohomology of \ec{(S,\psi)} and the Lie algebra cohomology of \ec{\mathscr{G}} in \eqref{eq:GLieAlg} with coefficients in the adjoint representation are trivial in degree one; and 
        \item the Lie algebra $\mathscr{G}$ is centerless.
    \end{itemize}
\end{proposition}

Finally, as a consequence of Theorem \ref{teo:PartiallySplit}, we conclude that if the mapping $\zeta$ defined in \eqref{eq:ZetaMap} is not surjective for a Poisson submanifold \ec{(S,\psi)}, then its first-order jet is not partially split.

    \section{Proof of Main Results}\label{sec:proof}

In this section, we provide the proofs of Theorems \ref{teo:ShortExactSeq}, \ref{teo:H1primprim}, \ref{teo:gen_main} and \ref{teo:PartiallySplit}. To this end, we first present in detail certain notions and facts used in Section \ref{sec:MainResults}.

In what follows, let \ec{(S,\psi)} be the Poisson submanifold of the Poisson manifold \ec{(M,\pi)} with normal bundle $E$ and co-normal bundle \ec{E^{\ast} \simeq \T{S^\circ}}.

\paragraph{Infinitesimal Poisson Algebra.}
Here, we present in detail formula \eqref{eq:Bracket_PT} for the bracket of the infinitesimal Poisson algebra \ec{\PP=(\Cinf{\mathrm{aff}}(E),\cdot,\{\cdot,\cdot\}^{\mathrm{aff}})} of $S$. 

Recall that the cotangent Lie algebroid \ec{(\T^{\ast}M, \pi^{\sharp}, [\cdot,\cdot]_{\pi})} of \ec{(M,\pi)} 
admits a natural restriction to the Poisson submanifold $(S,\psi)$,
    \begin{equation*}
        \big( \T^{\ast}_{S}M, \pi^{\sharp}|_{S}, [\cdot,\cdot]_{S} \big),
    \end{equation*}
and the restriction mapping \ec{r:\T^{\ast}_{S}M \to \T^{\ast}S} leads to a short exact sequence of Lie algebroids over \ec{(S,\psi)}:
    \begin{equation}\label{eq:ExtAlgbd}
        0 \to E^{\ast} \longhookrightarrow \T^{\ast}_{S}M \overset{r}{\longtwoheadrightarrow} \T^{\ast}S \to 0.
    \end{equation}
Here, \ec{(\T^{\ast}S, \psi^{\sharp}, [\cdot,\cdot]_{\psi})} is the cotangent Lie algebroid  of \ec{(S,\psi)}. Moreover, the bracket of sections \ec{[\cdot,\cdot]_1} on \ec{\Gamma E^\ast} gives rise to the Lie algebra in \eqref{eq:GLieAlg}.

The choice of an exponential map \ec{\mathbf{e}:E \rightarrow M} induces a section \ec{h:\T^{\ast} S \to \T^{\ast}_{S}M} (Ehresmann connection) on \eqref{eq:ExtAlgbd}, which gives the following $h$-depending data (see, \cite[Chapter 5]{Mar13} and \cite[Section 4]{RuGaVo20}):
    \begin{itemize}
      \item a contravariant derivative (or \ec{\T^{\ast}S}-connection) on \ec{E^*},
            \begin{equation*}
                \D: \Gamma(\T^{\ast}S) \times \Gamma E^* \longrightarrow \Gamma E^*, \quad
                    (\alpha,\zeta) \longmapsto \D_{\alpha}\zeta := [h(\alpha),\zeta]_{S};
            \end{equation*}

      \item a \ec{E^*}-valued bivector field \ec{\K \in \Gamma(\wedge^{2}\T{S} \otimes E^*)},
            \begin{equation*}
                \K(\alpha,\beta) := [h(\alpha),h(\beta)]_{S} - h[\alpha,\beta]_{\psi}, \quad \alpha, \beta \in \Gamma(\T^{\ast}S).
            \end{equation*}
    \end{itemize}
The data $\D$, $\K$ and $[\cdot,\cdot]_{1}$ satisfy the relations \cite{RuGaVo20}
    \begin{align}
            & \hspace{0.89cm} \D_{\alpha}[\eta,\xi]_{1} = [\D_{\alpha}\eta,\xi]_{1} + [\eta,\D_{\alpha}\xi]_{1}, \label{eq:PT1} \\
            & \hspace{0.89cm} \mathrm{Curv}^{\D}(\alpha,\beta) = [\K(\alpha,\beta),\cdot]_{1}, \label{eq:PT2} \\
            & \underset{(\alpha,\beta,\gamma)}{\mathfrak{S}} \D_{\alpha}\,\mathscr{K}(\beta,\gamma) + \mathscr{K}(\alpha,[\beta,\gamma]_{\psi}) = 0, \label{eq:PT3}
    \end{align}
for all \ec{\alpha,\beta,\gamma \in \Gamma(\T^{\ast}{S})} and \ec{\eta,\xi \in \Gamma{E^{\ast}}}.
Here,
    \begin{equation*}
        \mathrm{Curv}^{\D}(\alpha,\beta) := \D_\alpha\D_\beta -  \D_\beta\D_\alpha - \D_{[\alpha,\beta]_{\psi}}
    \end{equation*}
is the \emph{curvature} of $\D$ and ${\mathfrak{S}}$ denotes cyclic sum. 

Relations \eqref{eq:PT1}-\eqref{eq:PT3} allow us to define the infinitesimal Poisson algebra structure of $S$ by formula \eqref{eq:Bracket_PT}. As mentioned above, varying the exponential map yields isomorphic infinitesimal Poisson algebras.

\begin{remark}
The short exact sequence \eqref{eq:ExtAlgbd} is a particular case of a Lie algebroid extension, and the couple \ec{(\D, \K)} is the corresponding parameterization data (see \cite[Section 2]{Bra10}) (see also \cite{Mac87, Mac88}).
\end{remark}

\paragraph{Induced Cochain Complex.}
Let us describe the cochain complex \ec{\big( \mathfrak{Z}^{\bullet},\partial_{\D} \big)} in \eqref{eq:Cochain} arising in the description of the first cohomology of the infinitesimal Poisson algebra of $S$ given in Theorem \ref{teo:ShortExactSeq}.

The contravariant derivative $\D$ on $E^{\ast}$ induces a contravariant differential 
    \begin{equation*}
        \dd_{\D}:\Gamma(\wedge^{\bullet}\T{S}\otimes E^{\ast})\to\Gamma(\wedge^{\bullet}\T{S}\otimes E^{\ast})
    \end{equation*}
by the formula
    \begin{multline}\label{eq:dQ}
         (\dd_{\D}Q)(\dd f_{0},\ldots, \dd f_{k}) := \sum_{i=0}^{k}(-1)^{i}\,\D_{\dd f_{i}} \big( Q(\dd f_{0}, \ldots, \widehat{\dd f}_{i},\ldots,\dd f_{k}) \big) \\ 
         + \sum_{0 \leq i < j \leq k} (-1)^{i+j}\, Q\big( \dd\,\psi(\dd f_{i},\dd f_{j}),\dd f_{0} \ldots, \widehat{\dd f}_{i},\ldots,\widehat{\dd f}_{j}, \ldots \dd f_{k} \big),
    \end{multline}
for \ec{Q \in \Gamma(\wedge^{k}\T{S}\otimes E^{\ast})} and \ec{f_{0},\ldots,f_{k} \in \Cinf{S}}. 

Denote the center of the Lie algebra \ec{\mathscr{G}} in \eqref{eq:GLieAlg} by
    \begin{equation*}
        Z_{\mathscr{G}} := \{\xi \in \mathscr{G} \mid [\xi, \cdot]_{1} = 0\}.
    \end{equation*}
Consider the graded \ec{\Gamma(\wedge^{\bullet}\T{S})}-module $\mathfrak{Z}^{\bullet} := \scalebox{1.2}{$\oplus$}_{k \in \mathbb{Z}}\,\Gamma (\wedge^{k}\T{S}\otimes Z_{\mathscr{G}})$ of $Z_{\mathscr{G}}$-\emph{valued multivector fields} on $S$.
By relation \eqref{eq:PT1}, the restriction $\partial_{\D}$ of the contravariant differential \ec{\dd_\D} to \ec{\mathfrak{Z}^{\bullet}} is well defined
and, by relation \eqref{eq:PT2}, we get that \ec{\dd_{\D}^{2}} vanishes on \ec{\mathfrak{Z^{\bullet}}},
leading us to the cochain complex in \eqref{eq:Cochain}.

\paragraph{Casimir Elements.}
Now, we present a description of the Casimir elements of the infinitesimal Poisson algebra of \ec{(S,\psi)}.

\begin{lemma}\label{lema:CasimP}
The Lie algebra of Casimir elements of $\PP$ is given by
    \begin{equation*}
        \mathrm{Casim}(\PP) \simeq \{ k \oplus \xi \in \Cinf{\mathrm{aff}}(E) \mid k \in \mathrm{Casim}(S,\psi),\ \D_{\dd{k}} + [\xi, \cdot]_{1} = 0,\ \K( \dd{k}, \cdot ) - \D\xi = 0 \}.
    \end{equation*}
In particular, 
    \begin{equation}\label{eq:0H0CasimS}
        \{0\} \oplus \mathrm{H}^{0}_{\partial_{\D}}\subseteq\mathrm{Casim}(S,\psi),
    \end{equation}
where \ec{\mathrm{H}^{0}_{\partial_{\D}} = \{\xi \in Z_{\mathscr{G}} \mid \partial_{\D}\xi = 0\}} is the zeroth cohomology of the cochain complex \ec{\big( \mathfrak{Z}^{\bullet},\partial_{\D} \big)} in \eqref{eq:Cochain}.
Moreover, the following conditions are equivalent:
    \begin{enumerate}[label=(\alph*)]
        \item \ec{\mathrm{Casim}(\PP) \simeq \mathrm{Casim}(S,\psi) \oplus \mathrm{H}^{0}_{\partial_{\D}}}. \label{item:LemaCasim1}
        \item \ec{\mathrm{Casim}(S,\psi) \oplus \{0\} \subseteq \mathrm{Casim}(\PP)}. \label{item:LemaCasim2}
        \item \ec{\D_{\alpha} = 0} and \ec{\K(\alpha, \cdot ) = 0}, for all \ec{\alpha \in \ker{\psi^{\sharp}}}. \label{eq:DKkernel}
    \end{enumerate}
Any of this conditions imply that 
    \begin{align}\label{eq:KerJ}
        &\mathscr{C}_{0}(\PP) = \{0\}, \nonumber \\
        &\ker{\mathrm{J}} = \big\{ [\K(\dd{k},\cdot) - \D_{(\cdot)}\xi] \mid k \in \mathrm{Casim}(S,\psi),\ \D_{\dd{k}} + [\xi,\cdot]_{1} = 0 \big\} = \{0\},
    \end{align}
where \ec{\mathscr{C}_{0}(\PP)} is the submodule of \ec{\mathbb{D}(\Gamma{E^{\ast}})} defined in \eqref{eq:C0} and \ec{\mathrm{J}:\mathrm{H}_{\partial_{\D}}^{1} \rightarrow \mathrm{H}^{1}(\PP)} is the mapping in \eqref{eq:J}.
\end{lemma}
\begin{proof}
By definition of Casimir element, the lemma's first statement, inclusion \eqref{eq:0H0CasimS} and the equivalence of items \ref{item:LemaCasim1}-\ref{eq:DKkernel} follow directly from formula \eqref{eq:Bracket_PT}.
Now, by definition, it is clear that item \ref{eq:DKkernel} implies \ec{\mathscr{C}_{0}(\PP)=\{0\}}. Finally, by the definition of $\mathrm{J}$ and formula \eqref{eq:Bracket_PT}, the first equality in \eqref{eq:KerJ} holds. Consequently, item \ref{eq:DKkernel} implies that \ec{Q=-\partial_{\D}\eta} for some \ec{\eta\in Z_{\mathscr{G}}}, from which the second equality in \eqref{eq:KerJ} follows.
\end{proof}

\paragraph{The Lie algebra \ec{\boldsymbol{\mathfrak{M}(\PP)}}.}
Recall that a \emph{derivation} of the Lie algebra \ec{\mathscr{G}=(\Gamma E^\ast,[\cdot,\cdot]_1)} in \eqref{eq:GLieAlg} is an $\mathbb{R}$-linear mapping \ec{\ell: \Gamma E^{\ast}\to \Gamma E^{\ast}} such that
    \begin{equation*}
        \ell[\eta,\xi]_{1}=[\ell\eta,\xi]_{1}+[\eta,\ell\xi]_{1}, \quad \eta,\xi \in \Gamma E^*.
    \end{equation*}
We denote by \ec{\mathrm{Der}(\mathscr{G})} the space of all $\mathbb{R}$-linear derivations of $\mathscr{G}$.

Let us recall a useful fact on the Poisson derivations of $\PP$ (see \cite[Corollary 7.8]{GaRuVo23}).

\begin{lemma}\label{lemma:PoissDer}
Every Poisson derivation $X$ of $\PP$ is of the form $X=X_Q+X_\delta$, where \ec{Q\in\Gamma(\T S\otimes E^\ast)} and \ec{\delta\in\mathbb{D}(\Gamma E^*)} satisfy
\begin{align}
    &\delta\in\mathrm{Der}(\mathscr{G}), \label{eq:PoissDer1} \\
    &[\delta,\D_{\dd{f}}] - \D_{\dd\,\dlie{u}f} = [Q(f),\cdot]_1, \label{eq:PoissDer3} \\
    &(\dd_{\D}Q)(\dd f,\dd g) = \delta\big( \K(\dd{f},\dd{g}) \big) - \K\big( \dd\,\dlie{u}f,\dd{g} \big) - \K\big( \dd{f},\dd\,\dlie{u}g \big), \label{eq:PoissDer2}
\end{align}
for all \ec{f,g\in\Cinf{S}}.
Here, \ec{u=\sigma_\delta} is the symbol of $\delta$ and \ec{\dd_{\D}} is the contravariant differential defined in \eqref{eq:dQ}. In partcular, $u$ is a Poisson vector field of \ec{(S,\psi)}.
\end{lemma}

This lemma motivates the definition of the subspace
    \begin{equation}\label{eq:M}
        \mathfrak{M}(\PP) \subset \mathbb{D}(\Gamma{E^{\ast}}) \cap \mathrm{Der}(\mathscr{G})
    \end{equation}
consisting of all \ec{\delta \in \mathbb{D}(\Gamma{E^{\ast}})} that
can be extended to a derivation of the Poisson algebra \ec{\PP}, in the sense that
    \begin{itemize}
        \item \ec{\delta \in \mathrm{Der}(\mathscr{G})}; and 
        \item there exists \ec{\widetilde{Q} \in \Gamma(\T{S}\otimes E^*)} such that \eqref{eq:PoissDer3} and \eqref{eq:PoissDer2} hold for \ec{Q=\widetilde{Q}}.
    \end{itemize}

\begin{lemma}\label{lemma:adjoint}
Inner derivations of the Lie algebra $\mathscr{G}$ belong to \ec{\mathfrak{M}(\PP)}.
\end{lemma}
\begin{proof}
Given \ec{\xi\in\Gamma E^{\ast}}, consider the adjoint mapping \ec{\delta:=[\xi,\cdot]_1} and define \ec{\widetilde{Q}\in\Gamma(\T S\otimes E^\ast)} by $\widetilde{Q}(f):=-\D_{\dd{f}}\xi$. Since $[\cdot,\cdot]_1$ is a Lie bracket, property \eqref{eq:PoissDer1} is satisfied.
Moreover, \eqref{eq:PoissDer2} is consequence of relation \eqref{eq:PT2} since \ec{\sigma_{\delta}=0}. Finally, \eqref{eq:PoissDer3} follows from \eqref{eq:PT1}.
\end{proof}

On the other hand, for each \ec{\theta \in \Gamma\T^{\ast}S}, the derivative endomorphism $\D_\theta\in\mathbb{D}(\Gamma E^\ast)$ satisfies \eqref{eq:PoissDer1}, due to \eqref{eq:PT1}.
The following lemma explicitly describes the obstructions to extending $\D_\theta$ to a Poisson derivation of $\PP$:

\begin{lemma}\label{lema:dThetaSharpAlpha}
For every \ec{\theta \in \Gamma\T^{\ast}S},  the derivative endomorphism \ec{\D_{\theta} \in \mathbb{D}(\Gamma E^{\ast})} satisfies \ec{\sigma_{\D_{\theta}}=\psi^{\sharp}\theta}.
Moreover, we have \ec{\D_\theta\in\mathfrak{M}(\PP)} if the following conditions hold:
    \begin{equation}\label{eq:dTheta}
        \D_{\dd\theta(\psi^{\sharp}\dd{f}, \cdot)}=0 \quad \text{and} \quad  \K\big( \dd\theta(\psi^{\sharp}\dd{f}, \cdot), \cdot\big)=0.
    \end{equation}
\end{lemma}

\begin{proof}
By definition of contravariant derivative, we have
    \begin{equation*}
        \D_{\theta}(f\eta) = f\eta + (\dlie{\psi^{\sharp}\theta}f)\eta, \quad f \in \Cinf{S}, \eta \in \Gamma{E^{\ast}}.
    \end{equation*}
So, \ec{\D_{\theta} \in \mathbb{D}(\Gamma E^{\ast})} and \ec{\sigma_{\D_{\theta}}=\psi^{\sharp}\theta}. 
Now, by relations \eqref{eq:PT2} and \eqref{eq:PT3} and straightforward computations, one can show that \ec{\D_{\theta}} satisfies 
    \begin{align*}
        \D_{\theta} \circ \D_{\dd f}  - \D_{\dd f} \circ \D_{\theta} - \D_{\dd\,\dlie{\psi^{\sharp}\theta}f} &= \big[ \widetilde{Q}(f),\cdot \big]_{1} + \D_{\dd\theta(\psi^{\sharp}\dd{f}, \cdot)}, \\
        (\D_{\theta} \circ \K)(\dd{f},\dd{g}) - \K \big( \dd\,{\dlie{\psi^{\sharp}\theta}f},\dd{g} \big) - \K \big( \dd{f},\dd\,{\dlie{\psi^{\sharp}\theta}g}\big) &= \big( \dd_{\D}\widetilde{Q} \big)(\dd f,\dd g)\\
        &\phantom{=}\ + \K\big( \dd\theta(\psi^{\sharp}\dd{g}, \cdot), \dd{f} \big) + \K\big( \dd\theta(\psi^{\sharp}\dd{f}, \cdot), \dd{g} \big), \nonumber
    \end{align*}
for all \ec{f,g \in \Cinf{S}},
where \ec{\widetilde{Q} := \K(\theta, \cdot) \in \Gamma(\T{S}\otimes E^*)}.
So, if \eqref{eq:dTheta} holds, then equations \eqref{eq:PoissDer1}-\eqref{eq:PoissDer3} are satisfied for $\delta=\D_{\theta}$ and $\widetilde{Q}$.
\end{proof}

\paragraph{Proof of Theorem \ref{teo:ShortExactSeq}.}\label{para:TeoShortExactSeq}
By relations \eqref{eq:PT1}-\eqref{eq:PT3} and direct computations, one can show that \ec{\mathfrak{M}(\PP)} is an $\mathbb{R}$-Lie algebra with the commutator. 
Moreover, Lemma \ref{lema:dThetaSharpAlpha} implies that \ec{\mathfrak{M}(\PP)} contains the space of all derivative endomorphism induced by $\D$ on exact 1-forms,
    \begin{equation}\label{eq:C}
        \mathscr{C}(\PP):=\{\D_{\dd{h}}  \mid h \in \Cinf{S}\}.
    \end{equation}
Indeed, for \ec{\delta=\D_{\dd{h}}} we have
    \begin{equation}\label{eq:SigmaHam}
        \sigma_{\D_{\dd{h}}}=\psi(\dd{h}, \cdot),
    \end{equation}
and we can take \ec{\widetilde{Q}=\K(\dd{h},\cdot)}.

Let \ec{\mathfrak{M}_{0}(\PP)} be the Lie ideal of \ec{\mathfrak{M}(\PP)} consisting of all \ec{\delta \in \mathfrak{M}(\PP)} such that \ec{\sigma_{\delta}=0},
    \begin{equation}\label{eq:M0}
        \mathfrak{M}_{0}(\PP):=\ker{(\sigma|_{\mathfrak{M}(\PP)})}.
    \end{equation}
Note that every element of \ec{\mathfrak{M}_{0}(\PP)} is a \ec{\Cinf{S}}-linear derivation of \ec{\mathscr{G}}. Moreover, \ec{\operatorname{Inn}{\mathscr{G}} \subseteq \mathfrak{M}_{0}(\PP)}, by Lemma \ref{lemma:adjoint}.
Also, define
    \begin{equation}\label{eq:C0}
        \mathscr{C}_{0}(\PP):=\mathfrak{M}_{0}(\PP)\cap \mathscr{C}(\PP) = \{\D_{\dd{k}} \mid k \in \mathrm{Casim}(S,\psi)\}.
    \end{equation}

\begin{lemma}\label{lema:M0}
We have the short exact sequence
    \begin{equation*}
        0 \,\to\, \frac{\mathfrak{M}_{0}(\PP)}{\mathscr{C}_{0}(\PP)+\operatorname{Inn}{\mathscr{G}}}
        \,\longrightarrow\, \frac{\mathfrak{M}(\PP)}{\mathscr{C}(\PP) + \operatorname{Inn}{\mathscr{G}}}
        \,\longrightarrow\, \frac{\operatorname{Im}{\sigma|_{\mathfrak{M}(\PP)}}}{\mathrm{Ham}(S,\psi)}
        \,\to\, 0.
    \end{equation*}
\end{lemma}
\begin{proof}
The canonical map
    \begin{equation*}
        \frac{\mathfrak{M}_{0}(\PP)}{\mathscr{C}_{0}(\PP)+\operatorname{Inn}{\mathscr{G}}} \ni
            [\delta]
            \,\longmapsto\,
            [\delta] \in \frac{\mathfrak{M}(\PP)}{\mathscr{C}(\PP) + \operatorname{Inn}{\mathscr{G}}}
    \end{equation*}
is well defined and injective since \ec{\mathfrak{M}_{0}(\PP)\subset \mathfrak{M}(\PP)} and \ec{\mathscr{C}_{0}(\PP)\subset \mathscr{C}(\PP)}. Finally, the \ec{\mathbb{R}}-linear mapping
\begin{equation*}
        \frac{\mathfrak{M}(\PP)}{\mathscr{C}(\PP)+\operatorname{Inn}{\mathscr{G}}} \ni
            [\delta]
            \,\longmapsto\,
            [\sigma_{\delta}] \in \frac{\operatorname{Im}{\sigma|_{\mathfrak{M}(\PP)}}}{\mathrm{Ham}(S,\psi)}
    \end{equation*}
is well defined and surjective since \ec{\sigma|_{\operatorname{Inn}{\mathscr{G}}}=0} and \ec{\sigma_{\D_{\dd {h}}} \in \mathrm{Ham}(S,\psi)} by \eqref{eq:SigmaHam}, for all \ec{h \in \Cinf{S}}.
\end{proof}

\begin{lemma}\label{lema:secH1}
We have the short exact sequence
    \begin{equation*}
        0 \,\longrightarrow\, \frac{\mathrm{H}_{\partial_{\D}}^{1}}{\ker{\mathrm{J}}}
        \,\longrightarrow\, \mathrm{H}^{1}(\PP)
        \,\longrightarrow\, \frac{\mathfrak{M}(\PP)}{\mathscr{C}(\PP) + \operatorname{Inn}{\mathscr{G}}}
        \,\longrightarrow\, 0.
    \end{equation*}
\end{lemma}
\begin{proof}
By definition of $\mathrm{J}$, the natural induced $\mathbb{R}$-linear mapping given by
    \begin{equation*}
        {\mathrm{H}_{\partial_{\D}}^{1}} \,/\, {\ker{\mathrm{J}}}
        \,\ni \big\{ [Q] + \ker{\mathrm{J}} \big\}
            \,\longmapsto\,
        [X_{Q}]
        \in \mathrm{H}^{1}(\PP),
    \end{equation*}
is well defined and injective. Moreover, by definition of \ec{\mathfrak{M}(\PP)}, and taking into account formula \eqref{eq:Bracket_PT}, one can show that the $\mathbb{R}$-linear mapping
    \begin{equation*}
        \mathrm{H}^{1}(\PP) \ni
            \big[ X = X_{\delta} + X_{Q} \big]
                \longmapsto
            [\delta] \in \frac{\mathfrak{M}(\PP)}{\mathscr{C}(\PP) + \operatorname{Inn}{\mathscr{G}}}
    \end{equation*}
is well defined and surjective.
\end{proof}

Therefore, Theorem \ref{teo:ShortExactSeq} follows from Lemmas \ref{lema:M0} and \ref{lema:secH1}.

\paragraph{Proof of Theorem \ref{teo:H1primprim}.}\label{para:TeoH1primprim}
Given \ec{[\psi^{\sharp}\theta] \in \mathrm{H}^{1}_{\mathrm{cot}}(S,\psi)'}, 
we have that conditions \eqref{eq:dTheta} hold, by definition of \ec{\mathrm{H}^{1}_{\mathrm{cot}}(S,\psi)'} in \eqref{eq:CohoLieSub}.
By Lemma \ref{lema:dThetaSharpAlpha}, this implies that $\D_\theta\in\mathfrak{M}(\PP)$ and $\sigma_{\D_\theta}=\psi^{\sharp}\theta$.
So, we have a Poisson derivation $X$ of $\PP$ given by \ec{X = X_{\K(\theta,\cdot)}+X_{\D_{\theta}}}
satisfying \ec{\zeta[X] = [\sigma_{\D_\theta}] = [\psi^{\sharp}\theta]} for the mapping \eqref{eq:ZetaMap}.
This proves item \ref{item:TeoImag2}. 
Item \ref{item:TeoImag1} follows from \ref{item:TeoImag2} and the first inclusion in \eqref{eq:CohoLieSub}. 
Finally, item \ref{item:cH1primprim} follows from the fact that 
every \ec{Z \in \mathrm{Poiss}(M,\pi)} tangent to $S$ induces a Poisson derivation of $\PP$ which decends to \ec{Z|_{S} \in \mathrm{Poiss}(S,\psi)}. 
Indeed, by means of a linearization procedure on $E$ \cite{RuGaVo20,GaRuVo23}, $Z$ induces a Poisson derivation \ec{Z^{(2)}} of $\PP$.
By Lemma \ref{lema:Der1}, it admits a decomposition \ec{Z^{(2)} = Z^{(2)}_{\delta} + Z^{(2)}_{Q}}.
Furthermore, $\delta$ is the Lie derivative along a fiberwise linear vector field \ec{\mathrm{var}_{S}Z} on $E$, 
called the \emph{first variation} of $Z$ at $S$ \cite{RuGaVo20},
    \begin{equation*}
        \delta \simeq \dlie{\mathrm{var}_{S}Z} \in \mathfrak{M}(\PP).
    \end{equation*}
Moreover, since \ec{\mathrm{var}_{S}Z\,|_{S} = Z|_{S}}, we have that 
    \begin{equation*}
        \sigma|_{\mathfrak{M}(\PP)}(\delta)=Z|_{S}.
    \end{equation*}
So, by definition of $\zeta$, the class \ec{[Z^{(2)}] \in \mathrm{H}^{1}(\PP)} is such that \ec{\zeta[Z^{(2)}] = [Z|_{S}]}.
This completes the proof.

\paragraph{Proof of Theorem \ref{teo:gen_main}.}\label{para:Teogen_main}
Recall that we have the inclusion \eqref{eq:Inclusion}, and that $\mathfrak{M}_{0}(\PP)/(\mathscr{C}_{0}(\PP) + \operatorname{Inn}{\mathscr{G}})$ also admits an inclusion to the first cohomology of $\mathscr{G}$ with coefficients in the adjoint representation. 
It is then clear that the vanishing of the cohomologies \ref{cond2:H1(P0)trivial}-\ref{cond2:centerless} in degree one implies that $\operatorname{Im}{\sigma|_{\mathfrak{M}(\PP)}}/\mathrm{Ham}(S,\psi)$, $\mathfrak{M}_{0}(\PP)/(\mathscr{C}_{0}(\PP) + \operatorname{Inn}{\mathscr{G}})$, and $\mathrm{H}_{\partial_{\D}}^{1}/\ker\mathrm{J}$ are all trivial.
So, the short exact sequences of the diagram in Theorem \ref{teo:ShortExactSeq}, must be trivial, implying that $\mathrm{H}^{1}(\PP)=0$.

    \subsection{The Partially Split Case}

We now prove Theorem \ref{teo:PartiallySplit}. We begin with the necessary definitions and preliminary results.

\paragraph{Compatible Data.}
We have shown that the image of the mapping $\zeta$ in \eqref{eq:ZetaMap} is central to describe the first cohomology of
the infinitesimal Poisson algebra of the Poisson submanifold \ec{(S,\psi)}, and to establishing necessary and sufficient conditions for its vanishing.
In our approach, we have considered the \emph{cotangential} cohomology classes of $\mathrm{H}^{1}(M,\psi)$, that is, classes with representatives of the form $\psi\theta$ where $\theta \in \Gamma T^*S$,
since the data $\K$ and $\D$ allow us to ``lift'' $\theta$ to a derivation $X$ of $\Cinf{\mathrm{aff}}(E)$ that projects to $\psi^{\sharp}\theta$ (see Lemma \ref{lema:dThetaSharpAlpha}). 
This is the main reason of why $\zeta$ is surjective when $(S,\psi)$ is symplectic, and why $\mathrm{H}^{1}(\PP)$ admits the splitting in Theorem \ref{teo:SurjSym} in this case.

For general Poisson vector fields $u\in\mathrm{Poiss}(S,\psi)$, not necessarily cotangent, to the best of our knowledge there is no general procedure for ``lifting'' $u$ to a Poisson derivation of $\PP$ which, by Lemma \ref{lemma:PoissDer}, is equivalent to construct a derivative endomorphism $\delta$ with $\sigma_\delta=u$ such that \eqref{eq:PoissDer1}-\eqref{eq:PoissDer3} hold.
In the case when $E^*$ is locally trivial as a bundle of Lie algebras, there exists a connection $\nabla:\Gamma\T S\to\mathbb{D}(\Gamma E^\ast)$ with values on $\mathrm{Der}(\mathscr{G})$. In this case, by setting $\delta:=\nabla_u$, we have $\sigma_\delta=u$ and condition \eqref{eq:PoissDer1} holds.

\begin{definition}\label{def:NablaComp}
Let $(S,\psi)$ a Poisson submanifold, $(\D,\K)$ some parametrization data of its infinitesimal Poisson algebra,
$\nabla:\Gamma\T S\to\mathbb{D}(\Gamma E^\ast)$ a connection with values on $\mathrm{Der}(\mathscr{G})$,
and $U\in\Gamma(\T S\otimes\T^\ast S\otimes E^\ast)$.
We say that the data $(\nabla,U)$ is \emph{almost-compatible} with $(\D,\K)$ if
for every Poisson vector field $v\in\mathrm{Poiss}(S,\psi)$ and $Q:=U(\cdot,v)$ we have
    \begin{align*}
    \D_{\dd f} \circ \nabla_{v} - \nabla_{v} \circ \D_{\dd f} + \D_{\dd(\dlie{v}f)} &= \big[ Q(f),\cdot \big]_{1}, \\
    (\nabla_{v} \circ \K)(\dd{f},\dd{g}) - \K \big( \dd\,\dlie{v}{f},\dd{g} \big) - \K \big( \dd{f},\dd\,\dlie{v}{g}\big) &= - \big( \dd_{\D}Q \big)(\dd f,\dd g),
    \end{align*}
for all $f,g \in \Cinf{S}$.
If, in addition, $\D=\nabla\circ\psi^{\sharp}$, and $\K(\alpha,\beta)=U(\alpha,\psi^{\sharp}\beta)$ for all $\alpha,\beta\in\Gamma\T^\ast S$, 
then we say that $(\nabla,U)$ are \emph{compatible} with $(\D,\K)$. 
\end{definition}

\begin{proposition}\label{prop:NablaSobre}
The mapping in \eqref{eq:ZetaMap} is surjective
if $(\D,\K)$ admits an almost-compatible data.
Furthermore, if $(\nabla,U)$ is compatible with $(\D,\K)$, then the mapping \ec{\tau:\mathrm{H}^{1}(S,\psi)\to\mathrm{H}^{1}(\PP)}, \ec{[v]\mapsto[X_{U(\cdot,v)}+X_{\nabla_v}]}
is well-defined and is a right inverse of $\zeta$.
\end{proposition}

\begin{proof}
  If $(\nabla,U)$ is almost-compatible, then by Lemma \ref{lemma:PoissDer} the derivation $X_{U(\cdot,v)}+X_{\nabla_v}$ is Poisson for every $v\in\mathrm{Poiss}(S,\psi)$,
  and satisfies $\zeta[X_{U(\cdot,v)}+X_{\nabla_v}]=[v]$.
  On the other hand, the compatibility condition implies for all $h\in\Cinf{S}$ that $[X_{U(\cdot,\psi^{\sharp}\dd{h})}+X_{\nabla_{\psi^{\sharp}\dd{h}}}]=[X_{\K(\cdot,\dd{h})}+X_{\D_{\dd{h}}}]=[\{h\oplus0,\cdot\}^{\mathrm{aff}}]=0$.
\end{proof}

This proposition holds if $(S,\psi)$ is symplectic, since in this case a compatible data can be given by
$\nabla_v:=\D_{(\psi^{\sharp})^{-1}v}$ and $U(\alpha,v)=\K(\alpha,(\psi^{\sharp})^{-1}v)$.
Furthermore, this is the case for Poisson submanifolds with partially split first-order jet (see the next paragraphs).
The question of whether this procedure can be applied, for instance, to regular submanifolds, is still open.

\paragraph{Poisson Submanifolds with Partially Split Jets.}
Let $(S,\psi)\hookrightarrow(M,\pi)$ be a Poisson submanifold.
with first-order jet given by the restriction \ec{\big( \T^{\ast}_{S}M, \pi^{\sharp}|_{S}, [\cdot,\cdot]_{S} \big)} 
of the cotangent Lie algebroid of $M$ to $S$ \cite{Mar12}.
It admits a Lie algebroid extension 
\begin{equation}\label{eq:jet}
    0 \to E^{\ast} \longhookrightarrow \T^{\ast}_{S}M \overset{r}{\longtwoheadrightarrow} \T^{\ast}S \to 0,
\end{equation}
where $r:\T^{\ast}_{S}M \to \T^{\ast}S$ is the restriction of 1-forms.
The conormal bundle $E^*\simeq\T S^{\circ}$ is a bundle of Lie algebras, 
so we have a Lie algebra $\mathscr{G}:=(\Gamma E^\ast,[\cdot,\cdot]_1)$.
Furthermore, we have a representation (flat $\T^{\ast}_{S}M$-connection) on $E^\ast$ given by $a\mapsto\operatorname{ad}_{a}|_{E^\ast}$.

Recall from \cite{FerMar22} that the first-order jet \ec{\big( \T^{\ast}_{S}M, \pi^{\sharp}|_{S}, [\cdot,\cdot]_{S} \big)}
of a Poisson submanifold $(S,\psi)\hookrightarrow(M,\pi)$ is said to be \emph{partially split} 
if it admits an infinitesimally multiplicative (IM) connection 1-form,
that is, a pair $(L,l)$ consisting of an $\R$-linear mapping $L:\Gamma(\T^{\ast}_{S}M)\to\Gamma(\T^\ast S\otimes E)$ and a vector bundle map $l:\T^{\ast}_{S}M\to E^\ast$ with $l|_{E^\ast}=\mathrm{Id}_{E^\ast}$
satisfying (see \cite[Definition A6]{FerMar22})
\begin{align}\label{eq:PartiallySplit1}
  L(fa)=fL(a)+\dd{f}\otimes l(a), \qquad
  L[a,b]_S = \LL_aL(b) - \LL_bL(a), \qquad
  l[a,b]_S = \LL_al(b) - \ii_{\pi^{\sharp}|_{S}(b)}L(a)
\end{align}
for $a,b\in\Gamma(\T^{\ast}_{S}M)$, where the mapping \ec{\LL_a:\Gamma(\wedge^\bullet\T^\ast S\otimes E)\to\Gamma(\wedge^\bullet\T^\ast S\otimes E)} is
given for every $\gamma\in\Gamma(\wedge^k\T^\ast S\otimes E)$, and $v_1,\ldots,v_k\in\Gamma(\T S)$ by
    \begin{equation*}
        \LL_a\gamma(v_1,\ldots,v_k) := [a,\gamma(v_1,\ldots,v_k)]_S - \sum_{i=1}^{k}\gamma\left(v_1,\ldots,\left[\pi^{\sharp}|_{S}(a),v_i\right],\ldots,v_k\right)
    \end{equation*}
Note that the map $l$ is a (left) splitting of \eqref{eq:jet}, called the \emph{symbol} of $(L,l)$.

Following \cite[Section 5.4]{FMEhresmann}, recall that the IM connection 1-form $(L,l)$ 
gives rise to a connection $\nabla:\Gamma \T S \times \Gamma E^\ast\to\Gamma E^\ast$,
\begin{equation*}
\nabla_v\eta := \ii_vL(\eta), \qquad v\in\Gamma(\T S), \, \eta\in\Gamma E^\ast,
\end{equation*}
and tensor field $U\in\Gamma(\T S\otimes\T^\ast S\otimes E^\ast)$ by
\begin{equation*}
  U(\alpha,v):=-\ii_{v}L(h^l(\alpha)), \qquad \alpha\in\T^\ast S, v\in\T S,
\end{equation*}
where $h^l:\T^\ast S\to\T_S^\ast M$ is the right splitting of \eqref{eq:jet} induced by $l$: $r\circ h^l = \mathrm{Id}_{\T^\ast S}$, and $\mathrm{Im}(h^l)=\ker(l)$.
The pair $(\nabla,U)$ is called the \emph{coupling data} induced by $(L,l)$.

On the other hand, consider the parametrization data $(\D,\K)$ induced by $h^l$,
\begin{align*}
  \D:\Gamma\T^\ast S\times\Gamma E^\ast \to \Gamma E^\ast, &\qquad \D_\alpha\xi := [h^l(\alpha),\xi]_S, \\
  \K\in\Gamma(\wedge^2\T S\otimes E^\ast), &\qquad \K(\alpha,\beta) := [h^l(\alpha),h^l(\beta)]_S - h^l[\alpha,\beta]_{\psi}.
\end{align*}
The parametrization and coupling data of $(L,l)$ relate as follows:

\begin{lemma}\label{lemma:PartialSplitData}
  For $\xi\in\Gamma E^\ast$, and $\alpha,\beta\in\Gamma\T^\ast S$, we have
  \begin{equation}\label{eq:PartialSplitData}
    \D_\alpha\xi = \nabla_{\psi^{\sharp}\alpha}\xi, \qquad\text{and}\qquad \K(\alpha,\beta)=U(\alpha,\psi^{\sharp}\beta).
  \end{equation}
\end{lemma}

\begin{proof}
  Since $l|_{E^\ast}=\mathrm{Id}_{E^\ast}$, and $\D_\alpha\xi\in\Gamma E^\ast$, we have from \eqref{eq:PartiallySplit1} that
  \[
    \D_\alpha\xi = l(\D_\alpha\xi) = -l[\xi,h^l(\alpha)]_S = -l\left(\LL_\xi l(h^l(\alpha)) + \ii_{\pi^{\sharp}|_{S}h^l(\alpha)}L(\xi)\right) = \ii_{\psi^{\sharp}\alpha}L(\xi) =\nabla_{\psi^{\sharp}\alpha}\xi,
  \]
  where we also have used the following facts: $l\circ h^l=0$ and $\pi|_{S}^{\sharp}\circ h = \psi^{\sharp}$.
  Similarly, 
  \begin{align*}
    \K(\alpha,\beta) = l(\K(\alpha,\beta)) &= l[h^l(\alpha),h^l(\beta)]_S\\ 
    &= \LL_{h^l(\alpha)}(l(h^l(\beta)))-\ii_{\pi|_{S}^{\sharp}h^l(\beta)}L(h^l(\alpha)) = -\ii_{\psi^{\sharp}\beta}L(h^l(\alpha)) = U(\alpha,\psi^{\sharp}\beta).
  \end{align*}
\end{proof}

In other words, Lemma \ref{lemma:PartialSplitData} means that 
the parametrization data of a partially split first-order jet of a Poisson submanifold can be chosen as the restriction of some coupling data. 
Furthermore, the coupling data satisfy the following \emph{structure equations} (see \cite[Proposition 5.11]{FMEhresmann}), which imply \eqref{eq:PT1}-\eqref{eq:PT3} for the parametrization data:

\begin{proposition}\label{prop:StructureEquations}
  The coupling data $(\nabla, U)$ satisfy:
  \begin{align}
      &\nabla_v[\eta,\xi]_1 = [\nabla_v\eta,\xi]_1 + [\eta,\nabla_v\xi]_1, \label{eq:Structure1} \\
      &\mathrm{Curv}^{\nabla}(\psi^{\sharp}\alpha,v) = [U(\alpha,v),\cdot]_1, \label{eq:Structure2} \\
      &U([\alpha,\beta]_\psi,v) = \nabla_{\psi^{\sharp}\alpha}\,U(\beta,v) - \nabla_{\psi^{\sharp}\beta}\,U(\alpha,v) \label{eq:Structure3} + \nabla_{v}\,U(\alpha,\psi^{\sharp}\beta) + U(\alpha,[\psi^{\sharp}\beta,v]) - U(\beta,[\psi^{\sharp}\alpha,v]),
  \end{align}
  for all $\eta,\xi\in\Gamma E^\ast$, and $v\in\Gamma\T S$.
\end{proposition}

Equation \eqref{eq:Structure1} means that the connection $\nabla:\Gamma(\T S)\to\mathbb{D}(\Gamma E^\ast)$ 
values on derivations of the Lie algebra $\mathscr{G}:=(\Gamma E^\ast,[\cdot,\cdot]_1)$,
and \eqref{eq:Structure2} that the curvature of $\nabla$ values on $U$-adjoint operators of $\mathscr{G}$ if an argument is cotangential.
As a consequence, we have the following useful fact:

\begin{corollary}\label{cor:PartialSplitCompat}
  The coupling data $(\nabla,U)$ is compatible with $(\D,\K)$ in the sense of Definition \ref{def:NablaComp}.
\end{corollary}

\begin{proof}
  Recall that $[v,\psi^\sharp\dd{f}] = \psi^\sharp\dd{\dlie{v}f}$, for $v\in\mathrm{Poiss}(S,\psi)$ and $f\in\Cinf{S}$.
  By \eqref{eq:Structure2}, we have
  \[
    [U(\dd{f},v),\cdot]_1 = \nabla_{\psi^{\sharp}\dd{f}}\circ\nabla_v - \nabla_v\circ\nabla_{\psi^{\sharp}\dd{f}} - \nabla_{[\psi^{\sharp}\dd{f},v]}
     = \D_{\dd{f}}\circ\nabla_v - \nabla_v\circ\D_{\dd{f}} + \D_{\dd{\dlie{v}f}}.
  \]
  On the other hand, by setting $Q:=U(\cdot,v)$, we have from \eqref{eq:Structure3} that
  \begin{align*}
    -\left(\dd_{\D}Q\right)(\dd{f},\dd{g}) &= -\D_{\dd{f}}\left(Q(\dd g)\right) + \D_{\dd{g}}\left(Q(\dd f)\right) + Q(\dd\{f,g\}_\psi) \\ 
    &= -\nabla_{\psi^{\sharp}\dd{f}}U(\dd g,v) + \nabla_{\psi^{\sharp}\dd{g}}U(\dd f,v) + U([\dd{f},\dd{g}]_\psi,v) \\
    &=\phantom{-} \nabla_{v}\,U(\dd{f},\psi^{\sharp}\dd{g}) + U(\dd{f},[\psi^{\sharp}\dd{g},v]) - U(\dd{g},[\psi^{\sharp}\dd{f},v]) \\
    &= \phantom{-} \nabla_{v}\,U(\dd{f},\psi^{\sharp}\dd{g}) - U(\dd{f},\psi^{\sharp}\dd{\dlie{v}g}) + U(\dd{g},\psi^{\sharp}\dd{\dlie{v}f}) \\
    &=\phantom{-} (\nabla_{v} \circ \K)(\dd{f},\dd{g}) - \K \big( \dd\,\dlie{v}{f},\dd{g} \big) - \K \big( \dd{f},\dd\,\dlie{v}{g}\big).
  \end{align*}
\end{proof}

\paragraph{Proof of Theorem \ref{teo:PartiallySplit}.}\label{para:TeoPartialSplit}
The surjectivity of \ec{\zeta:\mathrm{H}^{1}(\PP)\to\mathrm{H}^1(S,\psi)} follow from Proposition \ref{prop:NablaSobre} and Corollary \ref{cor:PartialSplitCompat}.
Finally, since $\zeta$ is surjective, we conclude from Theorem \ref{teo:ShortExactSeq} that
        \[
          \mathrm{H}^{1}(\PP) \simeq \mathrm{H}^1(S,\psi) \oplus \mathrm{H}_{\partial_{\D}}^{1}\oplus\frac{\mathfrak{M}_{0}(\PP)}{\operatorname{Inn}{\mathscr{G}}}.
        \]
This proves Theorem \ref{teo:PartiallySplit}.


\end{document}